\documentclass[12pt]{amsart}
\usepackage[top=1in, bottom=1in, left=1in, right=1in]{geometry}

\usepackage{amssymb,mathrsfs,amsmath}

\newtheorem{theorem}{Theorem}[section]
\newtheorem{lemma}[theorem]{Lemma}
\newtheorem{cor}[theorem]{Corollary}
\newtheorem{prop}[theorem]{Proposition}
\theoremstyle{remark}
\newtheorem{remark}[theorem]{Remark}

\newcommand {\SN} {{\mathbb N}}
\newcommand {\SR} {{\mathbb R}}

\newcommand {\SZ} {{\mathbb Z}}

\newcommand {\SC} {{\mathbb C}}
\newcommand {\SU} {{\mathbb U}}
\newcommand {\SD} {{\mathbb D}}

\newcommand{\eq}[2]{ \begin{equation} \label{#1}\begin{split} #2 \end{split} \end{equation} }
\newcommand{\al}[1]{\begin{align} #1 \end{align} }
\newcommand{\als}[1]{\begin{align*} #1 \end{align*} }

\newcommand{\fl}[1]{\left\lfloor#1\right\rfloor}

\newcommand{\nn}{\nonumber \\}
\newcommand{\ds}{\displaystyle}

\newcommand{\bv}\boldsymbol{}

\numberwithin{equation}{section}

\begin{document}

\title{On multiplicative functions which are small on average}
\author{Dimitris Koukoulopoulos}
\address{D\'epartement de math\'ematiques et de statistique\\
Universit\'e de Montr\'eal\\
CP 6128 succ. Centre-Ville\\
Montr\'eal, QC H3C 3J7\\
Canada}
\email{{\tt koukoulo@dms.umontreal.ca}}


\date{\today}

\begin{abstract} Let $f$ be a completely multiplicative function that assumes values inside the unit disc. We show that if $\sum_{n\le x}f(n)\ll x/(\log x)^A$, $x\ge2$, for some $A>2$, then either $f(p)$ is small on average or $f$ {\it pretends to be} $\mu(n)n^{it}$ for some $t$.
\end{abstract}

\maketitle

\tableofcontents


\section{Introduction}\label{intro} A multiplicative function is an arithmetic function $f:\SN\to\SC$ which satisfies the functional equation $f(mn)=f(m)f(n)$ whenever $(m,n)=1$. Many central problems in number theory such as questions about the distribution of prime numbers can be phrased in terms of the average behaviour of multiplicative functions. A question of particular importance is when a given multiplicative function $f$ has mean value 0. This problem was solved by Hal\'asz \cite{hal2,hal3} when $f$ assumes values inside the unit circle $\SU=\{z\in\SC:|z|\le1\}$. His result states that unless $f$ \textit{pretends to be} $n^{it}$ for some $t\in\SR$, in the sense that 
\[
\sum_p \frac{1- \Re( f(p) p^{-it}) }  p <   \infty,
\]
then $f$ is 0 on average; the converse is also true. Hal\'asz also gave a quantitative version of his result and various authors (\cite{M78}, \cite{gs1}, \cite{T}) improved on it. The state of the art on this problem is Theorem \ref{halasz} below. Here and for the rest of this paper, given two multiplicative functions $f,g:\SN\to\SU$ and $x\ge y\ge1$, we set 
\[
\SD(f,g;y,x) = \left(\sum_{y<p\le x}\frac{1-\Re(f(p)\overline{g(p)})}p\right)^{1/2}.
\]
This quantity measures a certain ``distance'' between $f$ and $g$; as a matter of fact, it satisfies the triangle inequality (see Lemma \ref{triangle}).


\begin{theorem}\label{halasz} Let $f:\SN\to\SU$ be a multiplicative function and consider $x\ge1$ and $T\ge1$. Then we have that 
\[
\frac1x \sum_{n\le x}f(n)\ll   \frac{M_f(x;T)+1}{e^{M_f(x;T)}}   +\frac1T,
	\quad\text{where}\quad
	 M_f(x;T)=\min_{|t|\le T}\SD^2(f(n),n^{it};1,x).
\]
\end{theorem}


The generality of the above theorem is quite striking as it makes no assumptions for $f$ other than that its range of values is $\SU$. Nevertheless, the breadth of applicability of Theorem \ref{halasz} comes at a price: it can be shown that $M_f(x,T)\le\log\log x+O(1)$ (\cite{gs}), so the best bound on $\left|\sum_{n\le x}f(n)\right|$ that Theorem \ref{halasz} can yield is $O(x\log\log x/\log x$, where $c$ is some absolute constant. In the converse direction, Montgomery and Vaughan \cite{MV01} constructed for every $x\ge2$ a multiplicative function whose partial sum up to $x$ is of size $x\log\log x/\log x$, thus showing that Theorem \ref{halasz} is best possible. More recently, Granville and Soundararajan \cite{gs1} showed an explicit version of Theorem \ref{halasz} and constructed multiplicative functions whose summatory function achieves the bound in \cite{gs1} within a factor of 10. It is not very hard to construct slightly weaker but still almost extremal examples. Indeed, for every $M\ge3$, the completely multiplicative function $f$ defined by 
\eq{example}{
f(p) =\begin{cases}		
		1 	&\text{if}\ y<p\le 2y,\cr 
		0	&\text{otherwise}
	\end{cases}
}
satisfies the estimates 
\[
\sum_{n\le x}f(n) \ll \frac{x}{\log x} \prod_{p\le x}\left(1+\frac{f(p)}{p}\right)\ll \frac{x}{\log x}
	\quad\text{and}\quad 
M_f(x,T)=\log\log x+O(1)
\]
for all $x\ge3$ and $T\ge1$, the first one of which is a consequence of \cite[Theorem 5, p. 308]{T}. Moreover, when $x\in(3y/2,2y]$, we also have that
\[
\sum_{n\le x}f(n)=1+\sum_{y<p\le x}1	
	\sim	\frac x{\log x} - \frac{y}{\log y} 
	\asymp \frac{x}{\log x} 
	\asymp \frac{x}{e^{M_f(x;T)}}.
\]

Even though Theorem \ref{halasz} is optimal in this general setting, there are specific multiplicative functions whose partial sums satisfy (or are conjectured to satisfy) much sharper estimates than $\ll x\log\log x/\log x$. An important example is the M\"obius function, since controlling the size of its partial sums corresponds to estimating the error term in the prime number theorem. In order to understand the limitations of Hal\'asz's theorem better we study the following question: for which multiplicative functions $f$ there is a constant $A>0$ such that
\eq{small}{
\sum_{n\le x}f(n)\ll_A\frac x{(\log x)^A},
} 
for all $x\ge2$? For simplicity and in order to avoid technical issues at the prime 2, we assume further that $f$ is completely multiplicative (see Remark \ref{rk1} for further discussion about this assumption). The key observation towards understanding this problem is that if $f(p)$ is equal to $v\in\SU$ on average, then by the Selberg-Delange method \cite[Chap. II.5]{T} we expect that 
\eq{seldel}{
\sum_{n\le x}f(n) =  \left(  \frac{c_{f,v}}{\Gamma(v)} + o_f(1) \right)  x  (\log x)^{v-1}   \quad(x\to\infty),
} 
where $c_{f,v}$ is some non-zero constant and $\Gamma$ denotes Euler's Gamma function. Therefore, unless $v$ is a pole of $\Gamma$, relation\ \eqref{small} cannot hold for any $A>2\ge1-\Re(v)$. The only poles of $\Gamma$ in the unit circle are located at $-1$ and at $0$. If now $v=-1$, then $f$ looks like the M\"obius function $\mu$ which satisfies \eqref{small} by a quantitative form of the prime number theorem. Lastly, for the case $v=0$ Granville \cite{gs} showed that 
\eq{smallp}{
\sum_{p\le x}f(p)\log p  \ll		\frac x{(\log x)^B}	\quad(x\ge2)
	\qquad\implies\qquad 
\sum_{n\le x}f(n)  \ll   \frac x{(\log x)^B}\quad(x\ge2).
}

The above remarks seem to suggest that if \eqref{small} holds for some $A>2$, then the mean value of $f(p)$ has to be $-1$ or 0. However, this is obviously false, as the completely multiplicative function $(-1)^{\Omega(n)}n^{it}$ also satisfies \eqref{small} by the prime number theorem\footnote{$\Omega(n)$ denotes the number of prime divisors of $n$, counted with multiplicity.}. 
We make the refined guess that if \eqref{small} holds for some $A>2$, then either $f$ pretends to be $\mu(n)n^{it}$ for some $t$ or $f(p)$ is 0 on average. Theorem \ref{complex} below confirms our guess. 

We state our results taking into account the possibility that a multiplicative function might exhibit cancellation only past a certain point, say $Q$, which is related to the so-called analytic conductor of the associated $L$-function. So, instead of condition \eqref{small}, we assume that $f$ satisfies the estimate
\eq{small-Q}{
\left| \sum_{n\le x} f(n) \right| 
	\le \frac{x(\log Q)^{A-2}}{(\log x)^A} 	\quad(x\ge Q) ,
} 
for some constants $A\ge2$ and $Q\ge3$. Note that the condition that $A\ge2$ is necessary: if $A\in[0,1)\cup(1,2)$, then the completely multiplicative function $f(n)=(1-A)^{\Omega(n)}$ satisfies \eqref{small}, by \eqref{seldel}, but violates our guess: neither is $f(p)$ zero on average nor does $f$ pretend to be $\mu(n)n^{it}$ for some $t\in\SR$. Similarly, when $A=1$, the function given by \eqref{example} provides a counterexample to our guess.

Before we state Theorem \ref{complex}, we let, with a slight abuse of notation,
\eq{Qt}{
Q_t  =  \exp\left\{ 2(\log Q)(1+|t|)^{\frac{1}{A-2}} \right\} \ge Q^2
}
and
\eq{N(x;T)}{
N(x;T)&= \min_{ |t|\le T,\, Q_t\le x}
	\left\{   \log\log Q_t 
	+   \SD^2 ( f(n),\mu(n)n^{it}; Q_t,x) 	\right\} \\
	&\ge \max\left\{ \frac{  M_{\mu f}(x;T) }{2}  + O(1) , 2\log Q \right\}.
}

\begin{theorem}\label{complex}
Let $\epsilon>0$, $Q\ge 3$ and $f:\SN\to\SU$ be a completely multiplicative function that satisfies \eqref{small-Q} for some $A\ge2+\epsilon$.
\begin{enumerate}
\item 
For $x\ge Q^2$ and $T\ge1$, we have that
\[
\frac{1}{x} \sum_{p\le x}f(p)\log p 
	\ll_{A,\epsilon} (N(x;T) - \log\log Q)^2 \left(\frac{\log Q}{e^{N(x;T)}} \right)^{ B }  +  \frac{1}{T} ,
\]
where 
\eq{B}{
B = \begin{cases} 
 		(A-2)/(2A-2)		&\text{if}\ 2<A<3,\cr
		3(A-2)/(2A-2)		&\text{if}\ 3\le A<4,\cr
		2A/3-2 			&\text{if}\ A\ge 4.
	\end{cases}
}

\item Assume that $L(1+it_0,f)=0$ for some $t_0\in\SR$. Then we have that
\[
\frac{1}{x} \sum_{p\le x}( 1+ \Re( f(p)p^{-it_0})  )\log p	
	\ll_{A,\epsilon}  \left( \frac{\log Q_{t_0}}{\log x} \right)^{A-2} 	
	\quad (x\ge Q_{t_0})
\]
and, consequently\footnote{Here we use the Cauchy-Schwarz inequality and the fact that $|1+z|^2\le2(1+\Re(z))$ for $z\in\SU$.},
\[
\frac{1}{x}  \sum_{p\le x} | 1+ f(p)p^{-it_0}  | \log p	
	\ll_{A,\epsilon}\left( \frac{\log Q_{t_0}}{\log x} \right)^{\frac{A-2}{2}} 	
	\quad (x\ge Q_{t_0}).
\]
\end{enumerate}
\end{theorem}


\begin{remark}\label{rk3}
Using Theorem \ref{complex} and a similar argument with the one leading to \eqref{small}, it is possible to show that, given $f:\SN\to\SU$ as in Theorem \ref{complex}, we have that
\[
\frac{1}{x} \sum_{n\le x} \mu(n) f(n)
	\ll_{A,\epsilon}(N(x;T) - \log\log Q)^2 \left(\frac{\log Q}{e^{N(x;T)}} \right)^{B}  +  \frac{1}{T} 
\]
with $B$ defined by \eqref{B}. If $A>6$ and $Q=O(1)$, this constitutes an improvement over the estimate that Hal\'asz's theorem yields for the partial sums of summatory function of the multiplicative function $\mu f$, since $N(x;T)\ge M_{\mu f}(x;T)/2 +O(1)$, by relation \eqref{N(x;T)}. In fact, as we will see in Section \ref{complex-proof}, we may replace $B$ with the larger quantity $B'$, defined by \eqref{B'}. This yields an improvement over Hal\'asz's theorem as soon as $A>(31+\sqrt{681})/10=5.70959\dots$
\end{remark}


\begin{remark}\label{rk1} It is possible to weaken the condition that $f$ is a completely multiplicative function and extend Theorem \ref{complex} to the class of multiplicative functions $f:\SN\to\SU$, but we need a stronger assumption on $f$ than\ \eqref{small} that excludes a certain type of behavior of $f$ on powers of 2. To see that this is necessary, set $f(n)=1$ when $n$ is odd and $f(n)=-1$ when $n$ is even. Then $f$ is multiplicative and $\sum_{n\le x}f(n)=O(1)$. However, the conclusion of Theorem\ \ref{complex}(a) is clearly false.

In order to avoid the above example, we impose the stronger condition that 
\eq{small-Q-odd}{
\left| \sum_{ \substack{  n\le x,\, 2\nmid n }  } f(n) \right| 
	\le \frac{x(\log Q)^{A-2}}{(\log x)^A}	\quad(x\ge Q),
} 
which is clearly satisfied if $f$ is completely multiplicative and\ \eqref{small-Q} holds (possibly with $Q^{O(1)}$ in place of $Q$). Under condition\ \eqref{small-Q-odd}, Theorem\ \ref{complex} remains true. Indeed, set $\tilde{f}(n)=f(n)$ if $(n,2)=1$ and $\tilde{f}(n)=0$ otherwise. Also, let $g(n)=\prod_{p^a\|n}\tilde{f}(p)^a$ and write $g=\tilde{f}*h$, so that $h$ is supported on odd square-full integers and satisfies the bound $|h(n)|\le 2^{\Omega(n)}$ for all $n$. Note that if $n$ is square-full and not divisible by 2, 3, then $n\ge 5^{\Omega(n)} $. So 
\als{
\sum_{n\le x}  |h(n)|
	& \le \sum_{ \substack{ n\le x,\ 2\nmid n \\ n\ \text{square-full} } } 2^{ \Omega(n) } 
		= \sum_{ \substack{ 3^\nu \le x \\ \nu\ge2 }}  2^{\nu}
		\sum_{ \substack{ m\le x/3^\nu,\, (m,6)=1 \\ m\ \text{square-full} } } 2^{ \Omega(m) } \\
	&\le \sum_{ \substack{ 3^\nu \le x \\ \nu\ge2 }}  2^{\nu}
		\sum_{ \substack{ m\le x/3^\nu,\, (m,6)=1 \\ m\ \text{square-full} } } m^{\frac{\log 2}{\log 5}} 
		\ll \sum_{ \nu\ge2} 2^{\nu} \cdot 
		\left( \frac{x}{3^\nu} \right)^{\frac{1}{2}+\frac{\log2}{\log 5} } \ll x^{\frac{1}{2}+\frac{\log2}{\log 5} } \le x^{0.95} .
}
The above estimate and\ \eqref{small-Q-odd} imply that $g$ satisfies \eqref{small-Q} (with $Q^{O(1)}$ in place of $Q$), which allows us to apply Theorem \ref{complex} to it. Since $g( p ) = f( p )$ for all primes $p>2$, the conclusion of Theorem 1.2 holds for the function $f$ too, as claimed.

Similar extensions can be made to all subsequent results.
\end{remark}


When $f$ is real valued, it is possible to exclude the possibility that $f$ looks like $\mu(n)n^{it}$ for some $t\neq0$ (see Theorem \ref{real-complex}(b) below) and simplify the statement of Theorem \ref{complex}.


\begin{cor}\label{real}
Let $\epsilon>0$ and $Q\ge 3$. Consider a completely multiplicative function $f:\SN\to[-1,1]$ that satisfies \eqref{small-Q} for some $A\ge2+\epsilon$. If $L(1,f)\neq0$, then for $x\ge Q' = \exp\{(\log Q)\prod_{p>Q}(1-f(p)/p ) \}$ we have that
\[
\frac{1}{x} \sum_{p\le x} f(p) \log p 
	\ll_{A,\epsilon}\left(\log \frac{2\log x}{\log Q'} \right)
		 \left(\frac{\log Q'}{ \log x} \right)^{B} 
\]
with $B$ is defined as in \eqref{B}. On the other hand, if $L(1,f)=0$, then for $x\ge Q$ we have that
\[
\frac{1}{x} \sum_{p\le x}(1+ f(p) ) \log p \ll_{A,\epsilon} \left(\frac{\log Q}{\log x} \right)^{A-2}.
\]
\end{cor}

\begin{proof} For every $T\ge1$, Theorem \ref{real-complex} and Lemma \ref{dist-l1} below imply that 
\[
N(x; T) 
	\ge \log\log Q + \sum_{Q< p\le x} \frac{1+f(p)}{p} + O_\epsilon(1) 
	\ge \log\left( (\log x) \prod_{p>Q} \left( 1 - \frac{ f(p)} p \right)^{-1} \right)  + O_\epsilon(1) .
\]
So the first part of Corollary \ref{real} follows from Theorem \ref{complex}(a) applied with $T=\infty$. Finally, if $L(1,f)=0$, then the desired result is an immediate consequence of Theorem \ref{complex}(b).
\end{proof}


The dependence on $A$ can be made explicit in the above results. Keeping track of the implied constants leads to the following result, where we have assumed for simplicity that $f$ is real valued. A similar but weaker result holds in the general case of a complex valued function of modulus $\le1$.

\begin{theorem}\label{power} 
Let $\delta\in(0,1/3)$, $Q\ge e^{1/\delta}$ and $f:\SN\to[-1,1]$ be a completely multiplicative function such that
\eq{small-power}{
\left| \sum_{n\le x}f(n) \right|\le  \frac{ x^{1-\delta } }{ (\log x)^2 }   \quad(x\ge Q). 
}
\begin{enumerate}
\item If $L(1,f)\neq0$, then we have that
\[
\sum_{p\le x}f(p)\log p 
	\ll  \frac{x}{ e^{c\sqrt{\log x}} } +  x^{1-  c\eta /  \log Q  } \quad(x \ge Q),
\]
for some $c=c(\delta)$, where $\eta=\prod_{p>Q}(1 - f(p)/p)^{-1} \ll 1$. Moreover, there is a constant $c'\in(0,1)$ such that $L(s,f)$ has at most one zero in the interval $[1-c'/\log Q,1)$, say at $\beta$. If such a zero does not exist, we set $\beta=1-c'/\log Q$. In any case, we have that $\eta\asymp (1-\beta)\log Q$. 

\item If $L(1,f)=0$, then
\[
\sum_{p\le x}(1+f(p))\log p 
	\ll  x^{1-  1 / (61\log Q) }     \quad(x \ge Q).
\]
\end{enumerate}
\end{theorem}


It is evident from the above result that our methods are of comparable strength with more classical arguments that use the analyticity of $L(s,f)$ to the left of the line $\Re(s)=1$ such as the ones in \cite{Dav}. Indeed, in \cite{K} it was shown how to combine the methods of this paper with estimates for exponential sums due to Korobov and Vinogradov to give a new proof of the best error term known in the prime number theorem for arithmetic progressions. However, since we always work with conditions of the form \eqref{small-Q}, we are forced to use different methods than analytic continuation and the residue theorem. So the proofs of Theorems \ref{complex} and \ref{power} above, as well as of the results in \cite{K}, are `elementary' from a broad point of view. We will give a brief outline of the main ideas that go into them in Subsection \ref{overview} below.  

\medskip

Finally, it would be desirable to extend the results of this paper to multiplicative functions that assume values outside the unit circle too. A large portion of the paper can be generalized to multiplicative functions whose values at primes are uniformly bounded. However, the results of Section \ref{distances}, which are of key importance, cannot be transferred immediately.


\subsection{Notation}\label{notation} For an integer $n$ we denote with $P^+(n)$ and $P^-(n)$ the greatest and smallest prime divisors of $n$, respectively, with the notational convention that $P^+(1)=1$ and $P^-(1)=\infty$. For two arithmetic functions $f,g:\SN\to\SC$ we write $f*g$ for their Dirichlet convolution, defined by $(f*g)(n)=\sum_{ab=n}f(a)g(b)$. Also, for $s\in\SC$  and $y\ge1$ we set 
\[
L(s,f) = \sum_{n=1}^\infty \frac{f(n)}{n^s}
	\quad\text{and}\quad 
L_y(s,f) = \sum_{P^-(n)>y}\frac{f(n)}{n^s},
\]
provided that the series converge. In the special case that $f(n)=1$ for all $n$, we use the notation 
\[
\zeta_y(s)=L_y(s,1).
\] 
We let $\tau_k(n)=\sum_{d_1\cdots d_k=n}1$ and we denote with $\mu(n)$ the M\"obius function, defined to be $(-1)^{\#\{p|n\}}$ if $n$ is squarefree and 0 otherwise. Moreover, we recall the definition of the generalized von Mangoldt functions $\Lambda_k=\mu*\log^k$, $k\in\SN\cup\{0\}$. The case $k=1$ corresponds to the standard von Mangoldt function, which we denote simply by $\Lambda$; its value at an integer $n$ is $\log p$ if $n$ is a prime power $p^a$ and 0 otherwise. Finally, the notation $F\ll_{a,b,\dots}G$ means that $|F|\le CG$, where $C$ is a constant that depends at most on the subscripts $a,b,\dots$, and $F\asymp_{a,b,\dots} G$ means that $F\ll_{a,b,\dots}G$ and $G\ll_{a,b,\dots}F$. In general, we reserve the letters $c$ and $C$ in order to denote constants, not necessarily the same ones in every place, and possibly depending on certain parameters that will be specified using subscripts and other means.


\subsection{Overview of the proof methodology}\label{overview} In this subsection we outline the main ideas that go into the proof of Theorem \ref{small} (and, hence, of Theorem \ref{power}). We start by discussing the proof of part (a). The main restriction we need to overcome is that condition \eqref{small-Q} is not strong enough to guarantee the analytic continuation of  $L(s,f)$ to the left of the line $\Re(s)=1$. This renders arguments based on the location of zeroes of $L(s,f)$ inapplicable. Instead, we employ an idea used by Iwaniec and Kowalski \cite[p. 40-42]{IK} to give a new proof of the prime number theorem, upon which we improve by combining it with some ideas from sieve methods. Our starting point is the combinatorial identity
\eq{comb-id}{
\left(\frac{-F'}{F}\right)^{(k-1)}(s)=k!\sum_{a_1+2a_2+\cdots=k}\frac{(-1+a_1+a_2+\cdots)!}{a_1!a_2!\cdots}
\left(\frac{-F'}{1!F}(s)\right)^{a_1}\left(\frac{-F''}{2!F}(s)\right)^{a_2}\cdots,
}
which translates upper bounds on the the derivatives of $(L'/L)(s,f)$ to upper bounds on the derivatives of $L(s,f)$ and lower bounds on $|L(s,f)|$. Our assumption that $f$ satisfies \eqref{small-Q} then allows us to bound $L^{(j)}(s,f)$ easily. On the other hand, lower bounds on $|L(s,f)|$ with $s=1+1/\log x+ it$ are equivalent to lower bounds on the distance function $\SD^2(f(n),\mu(n)n^{it};1,x)$. This explains the appearance of the quantity $N(x;T)$ in the statement of our results. Then we use an inversion formula, such as Perron's inversion formula, to insert the information that we have obtained on $(L'/L)^{(k-1)}(s,f)$ and estimate $\sum_{n\le x}\Lambda(n) f(n) (\log n)^{k-1}$. A crucial role when handling the integral of $(L'/L)^{(k-1)}(s,f)$ is played by the fact that $L(s,f)$ cannot be too small too often. This is the context of Theorem \ref{min-dist} below. Finally, removing the extra factor $(\log n)^{k-1}$ is easily accomplished by partial summation.

\medskip

The above simple description of the argument contains at least two inaccuracies. Firstly, if we apply \eqref{comb-id} with $F(z)=L(z,f)$, we are bound to lose some logarithmic factors. The reason is that the partial sums of $f(n)n^{-it}$ are small only past a barrier, which is roughly equal to $Q_t$ (see relation \eqref{small-Qt}). Therefore, it is possible that $f(n)n^{-it}\approx 1$ for $n\le Q_t$, which would force $L^{(j)}(1+1/\log x+it,f)$ to be abnormally large. However, if this is the case, then $L(1+1/\log x+it,f)$ would be also abnormally large. In order to take this phenomenon into account, we perform a preliminary sieve. So, instead, we apply \eqref{comb-id} with $F(z) = L_{Q_t}(z,f)$ and $s=1+1/\log x +it $. This ensures that we only consider integers $n>Q_t$. An additional advantage in considering $L_{Q_t}(z,f)$ (instead of $\sum_{n>Q_t}f(n)/n^z$, for example) is that $L_{Q_t}(z,f)$ possesses an Euler product. In particular, we can relate the size of $L_{Q_t}(1+1/\log x+it,f)$ to the distance function $\SD^2(f(n),\mu(n);Q_t,x)$. This explains why our results are stated using $N(x;T)$ instead of $M_{\mu f}(x;T)$.

The second inaccuracy in our initial description of the proof of Theorem \ref{complex}(a) concerns the way we translate bounds on the derivatives of $(L'/L)(s,f)$ to bounds on the partial sums of $f(p)\log p$. Instead of applying Perron's inversion formula, which would cause a loss of some logarithmic factors when $x \le Q^{C}$, we use the fact that mean values of multiplicative functions obey certain recursive relations, which allows us to smoothen out potential irregularities. Indeed, if $g$ is a completely multiplicative function, then we have that
\eq{d-delay-1}{
\sum_{n\le x} g(n) \log n = \sum_{dm\le x} \Lambda(d)g(d) g(m),
}
as a consequence of the identity $\log=\Lambda*1$. Relation \eqref{d-delay-1} plays a prominent role in the study of averages of multiplicative functions. In particular, it is featured in the proof of Theorem \ref{halasz}. Taking $g=f$ in \eqref{d-delay-1}, and applying relation \eqref{small-Q} and Dirichlet's hyperbola method, we deduce that
\[
\sum_{m\le \sqrt{x}} f(m) \sum_{d\le x/m} f(d)\Lambda(d) \ll_A \frac{ x (\log Q)^{A-2}}{ (\log x)^{A-1} } \quad (x\ge Q^2).
\]
The summand corresponding to $m=1$ is $\sum_{d\le x}f(d)\Lambda(d)$, that is to say, the sum we are trying to bound. So
\eq{d-delay-2}{
\sum_{d\le x} f(d) \Lambda(d) 
	= - \sum_{1<m\le \sqrt{x}} f(m) \sum_{d\le x/m} f(d)\Lambda(d) 
	+ O_A  \left( \frac{ x (\log Q)^{A-2}}{ (\log x)^{A-1} } \right)     \quad (x\ge Q^2).
}
However, relation \eqref{d-delay-2} is not very useful as it stands because the summand with $m=2$ on its right hand side equals $f(2)\sum_{d\le x/2} f(d)\Lambda(d)$, a quantity which is likely to have roughly the same size as the `main term' $\sum_{d\le x}f(d)\Lambda(d)$. In order to overcome this obstacle, we resort to sieve methods again. Instead of letting $g=f$ in \eqref{d-delay-1}, we fix some parameter $z\ge Q$ and we let $g(n)=f(n)$ when $P^-(n)>z$ and $g(n)=0$ otherwise. Then relation \eqref{d-delay-1} and the fundamental lemma of sieve methods (see Lemma \ref{fund-lemma}) yield that\footnote{See the proof of Theorem \ref{complex}(a) in Subsection \ref{complex-completion}.}
\eq{d-delay-3}{
\sum_{\substack{ m\le \sqrt{x} \\ P^-(m)>z }} f(m) \sum_{d\le x/m} f(d)\Lambda(d) 
	\ll_A x \cdot \left( \frac{\log z}{\log x}\right)^{A-1} .
}
The summand corresponding to $m=1$ in \eqref{d-delay-3} is $\sum_{d\le x}f(d)\Lambda(d)$, as before. However, all the summands on \eqref{d-delay-3} with $m\in(1,z]$ vanish, so the problem we had with relation \eqref{d-delay-2} does not exist anymore (and there is the additional advantage that $m$ runs over a subset of the integers in $(z,\sqrt{x}]$ that has density $1/\log z$ instead of all integers in $(z,\sqrt{x}]$). Finally, we make use of relation \eqref{d-delay-3} in a similar fashion as in the proof of Hal\'asz's theorem to establish Theorem \ref{complex}(a). However, in Hal\'asz's theorem the factorization $F'=(F'/F) \cdot F$ was key, and in our case such a factorization is not available. This leads to employing a different strategy, as we will see in the proof of Proposition \ref{complex-prop-3}.

\medskip

Finally, we discuss briefly the proof of part (b) of Theorem \ref{complex}, which is distinctly different from and simpler than the proof of part (a). The following argument was pointed out to us by an anonymous referee and by Andrew Granville. For simplicity, we assume that $t_0=0$. Our starting point is the observation that the convolution $1*f*1*\overline{f}$ assumes non-negative real values. Therefore
\[
0\le 2\sum_{p\le x} (1+\Re(f(p))) \le \sum_{n\le x} (1*f*1*\overline{f}) (n).
\]
In order to handle the sum on the right hand side, we use the fact that $\sum_{n\le x}(1*f)(n)$ is small, a consequence of our assumption that $L(1,f)=0$ and of relation \eqref{small-Q}. So Dirichlet's hyperbola method implies that $\sum_{p\le x}(1+\Re(f(p))) $ is small too.

As in the proof of part (a), in order to obtain the actual statement of Theorem \ref{complex}(b), we need to be more careful. We start instead from the formula
\[
0\le 2\sum_{Q< p\le x} (1+\Re(f(p))) \le \sum_{ \substack{ n\le x \\ P^-(n)>Q }} (1*f*1*\overline{f}) (n).
\]
A crucial role in the proof is also played by Theorem \ref{complex-dist} below.


\subsection{Outline of the paper} We give here a brief description of how the paper is structured. Firstly, in Section \ref{mainresults} we state some additional main results, Theorems \ref{complex-dist}, \ref{min-dist}, \ref{real-complex} and \ref{siegel}, which are a bit more technical than Theorems \ref{complex} and \ref{power}. Section \ref{zeta-section} contains a series of auxiliary estimates concerning the Riemann $\zeta$ function and its derivatives. Subsequently, in Section \ref{l(s,f)} we establish various bounds for partial sums of multiplicative functions and derive from them estimates for high derivatives of $L(s,f)$. In Section \ref{distances} we state and prove several results related to distance of a multiplicative function from the M\"obius function and apply them to control the size of $L(s,f)$ close to the line $\Re(s)=1$. We also establish Theorem \ref{real-complex}. The results of Section \ref{distances} are then used in Section \ref{dist-proofs} to show Theorems \ref{complex-dist} and \ref{min-dist}. In turn, these two results play a crucial in the proof of part (b) of Theorems \ref{complex} and \ref{power}, which is given in Section \ref{L(1+it_0)=0}. Next, in Section \ref{l(1,f)} we see how to control the size of $L(1,f)$ in terms of a potential \textit{Siegel zero} and demonstrate Theorem \ref{siegel}. Finally, the proof of part (a) of Theorems \ref{complex} and \ref{power} is split among two sections. In Section \ref{1/l(s,f)} we prove some required bounds on high derivatives of $1/L(s,f)$ and $(L'/L)(s,f)$. The main part of the proof of these theorems is then given in Section \ref{complex-proof}.


\section{Additional main results}\label{mainresults}

In this section we state various additional results that our methods yield and that we believe are of independent interest. These results will also play a key role in the proof of Theorems\ \ref{complex} and\ \ref{power}. The proof of the two theorems below will be given in Section \ref{dist-proofs}.


\begin{theorem}\label{complex-dist} Let $\epsilon>0$ and $Q\ge3$. Let $f:\SN\to\SU$ be a completely multiplicative function that satisfies \eqref{small-Q} with $A=2+\epsilon$. Then there is some $Q'\in[Q,+\infty]$ such that 
\[
\sum_{Q<p\le Q'} \frac{|1+f(p)|}{p} \ll_\epsilon1
	\quad\text{and}\quad
\sum_{Q'<p\le z}\frac{f(p)}p  \ll_\epsilon1	\quad(z\ge Q').
\]
So, for any $y\ge Q$, we have that $|L_y(1,f)| \asymp_\epsilon (\log y)/\log(yQ')$. In particular, letting $y=Q$, we find that $\log Q'\asymp_\epsilon   (\log Q)/|L_Q(1,f)|$.
\end{theorem}


Here and for the rest of this paper we define 
\eq{Vt}{
V_t =  \exp\{ (\log(3+|t|))^{2/3} (\log\log(3+|t|))^{1/3} \} \quad(t\in\SR).
}


\begin{theorem}\label{min-dist} Let $f:\SN\to\SU$ be a completely multiplicative function, $\tau\ge1$, $Q\ge (V_{2\tau})^{100}$ and $\epsilon>0$ such that 
\[
\left| \sum_{n\le x}f(n)n^{-it} \right|
	\le \frac {x(\log Q)^\epsilon} { (\log x)^{2+\epsilon} } 
		\quad(x\ge Q,\,t\in[-\tau,\tau]).
\] 
Fix $\sigma>1$ and $J\subset[-\tau,\tau]$, and let $t_0\in J$ be such that $|L_Q(\sigma+it_0,f)|=\min_{t\in J}|L_Q(\sigma+it,f)|=:\eta$. Then, for any $t\in J$, we have that
\[
|L_Q(\sigma+it,f)|  \asymp_\epsilon 
	\begin{cases}	
		\eta & \text{if}\ |t-t_0| \le \eta/\log Q ,   \cr 
		|t-t_0|\log Q & \text{if}\ \eta/\log Q \le |t-t_0| \le 1/ \log Q ,  \cr  
		1 & \text{if}\ |t-t_0| \ge 1/\log Q  .
	\end{cases}
\]
\end{theorem}


If we have additional information about $f$, then it is possible to control the distance of $f$ from $\mu(n)n^{it}$. This is the context of the following theorem, which will be proven in Section \ref{distances}.

\begin{theorem}\label{real-complex} Fix $\epsilon>0$ and $Q\ge3$. Let $f:\SN\to\SU$ be a completely multiplicative function that satisfies \eqref{small-Q} with $A=2+\epsilon$.
\begin{enumerate}
    \item If $f^2$ satisfies \eqref{small-Q} too with $A=2+\epsilon$, then 
    	\[
	\sum_{Q<p\le x}\frac{f(p)}p \ll_\epsilon 1  	\quad(x\ge Q).
    	\]
    \item Assume that $f$ is real valued, and let $t\in\SR$ and $x\ge Q\ge V_{2t}^{100}$. If $|t|\ge1/\log Q$, then 
    	\[
	\sum_{Q<p\le x} \frac{f(p)} { p^{1+it}} \ll_\epsilon 1, 
	\]
	and if $|t|\le1/\log Q$, then 
	\als{
	\SD^2(f(n),\mu(n)n^{it};Q,x)  
		&\ge   \SD^2(f(n),\mu(n);Q,x) -  O_\epsilon (1) \\
		&\ge \log\left(\frac{\log x}{\log Q}\right) +  \log L_Q(1,f) -  O_\epsilon(1).
	}
\end{enumerate}
\end{theorem}


Finally, if we have at our disposal very good estimates on the summatory function of $f$, then partial summation implies that $L(s,f)$ converges to the left of the line $\Re(s)=1$. Moreover, the size of $L(1,f)$ can be determined using information on the location of zeroes of $L(s,f)$ in a neighbourhood around 1. Remaining faithful to the `elementary' nature of the paper, in Section \ref{l(1,f)} we give a proof of this fact in the case when $f$ is real valued that avoids the use of complex analytic tools. In doing so, we make use of some ideas of Pintz \cite{Pi76I,Pi76II,Pi76III}, who gave elementary proofs of some related results when $f$ is a real Dirichlet character.


\begin{theorem}\label{siegel} Let $Q\ge3$ and $f:\SN\to[-1,1]$ be a completely multiplicative function satisfying \eqref{small-power} with $\delta=1/\log Q$. Then $L(s,f)$ converges in the half plane $\Re(s)>1-1/\log Q$ and there is an absolute constant $c\in(0,1/2)$ such that $L(s,f)$ has at most one zero in $[1-2c/\log Q,1]$, say at $\beta$. If no such zero exists, we set $\beta=1-2c/\log Q$. In any case, there are positive constants $c_1$ and $c_2$ such that for all $\sigma\in[1-c/\log Q,1+c\log Q]$ we have that 
\[
c_1(\sigma-\beta)\log Q  \le L_Q(\sigma,f)   \le c_2(\sigma-\beta)\log Q.
\]
\end{theorem}


\section{Bounds for $\zeta$ and its derivatives}\label{zeta-section}

In this section we give some estimates about the Riemann $\zeta$ function. First, we have the following lemma, which is a simple corollary of Lemma 2.4 in \cite{K}.


\begin{lemma}\label{zeta} 
Let $y\ge2$ and $t\in\SR$ with $y\ge V_t^{100}$. For $x\ge y$, we have that 
\[
\sum_{\substack{n\le x \\ P^-(n)>y } }n^{it}
	= \frac{x^{1+it}}{1+it}  \prod_{p\le y}\left(1-\frac1p\right)
	+  O \left(   \frac{x^{1- 1/ (30\log y) } }  { \log y } \right) .
\]
Consequently, if $s=\sigma+it$ is such that $\sigma\ge1-1/(60\log y)$, then
\[
\sum_{\substack{n\le x\\P^-(n)>y}}\frac1{n^s}=\left(\frac{1-x^{-s+1}}{s-1}+\gamma_{s,y}\right)\prod_{p\le y}\left(1-\frac1p\right)+O\left( x^{1-\sigma- 1/(30\log y)} \right),
\]
where $\gamma_{s,y}$ is a constant that depends only on $s$ and $y$, it is real valued for $s\in\SR$, and it satisfies the uniform bound $\gamma_{s,y}\ll\log y$ for $s$ and $y$ as above.
\end{lemma}


\begin{proof} The first part of the lemma is a special case of Lemma 2.4 in \cite{K}. For the second part, note that
\[
\sum_{\substack{n\le x\\P^-(n)>y}}\frac1{n^s}
	= 1+ \int_y^x \frac{1}{u^\sigma} d\left( \sum_{\substack{n\le x\\P^-(n)>y}} n^{-it} \right) .
\]
So if we write
\[
\sum_{\substack{n\le x \\ P^-(n)>y } }n^{-it}
	= \frac{x^{1-it}}{1-it}  \prod_{p\le y}\left(1-\frac1p\right) + R(x,y,t),
\]
then the second part follows with
\[
\gamma_{s,y}
 	= \left( 1 - \frac{R(y,y,t)}{y^\sigma} +   \sigma \int_y^\infty \frac{R(u,y,t)}{u^{\sigma+1}} du \right) 
	\prod_{p\le y}\left( 1- \frac1p \right)^{-1}  - \int_1^y\frac{du}{u^s} .
\]
\end{proof}


In order to prove our next result on $\zeta$, we need a lemma due to Montgomery \cite[Theorem 3, p. 131]{M94}:

\begin{lemma}\label{mont} 
Let $A(s)=\sum_{n\ge1}a_n/n^s$ and $B(s)=\sum_{n\ge 1}b_n/n^s$ be two Dirichlet series which converge for $\Re(s)>1$. If $|a_n|\le b_n$ for all $n\in\SN$, then 
\[
\int_{-T}^T  |A(\sigma+it)|^2  dt  \le  3  \int_{-T}^T  |B(\sigma+it)|^2  dt    \quad(\sigma>1,\,T\ge0).
\]
\end{lemma}


Using the two lemmas above, we shall prove the following estimate which concerns averages of $\Lambda_k=\mu*\log^k$ and, consequently, provides estimates for $\zeta^{(k)}/\zeta$.

\begin{lemma}\label{lambda} Let $x,z\ge3$, $k\in\SN\cup\{0\}$, $m\in\SN$, $r\in\SN\cup\{0\}$ and $T\ge2$. There exists a constant $c>0$ such that
\[
\sum_{P^+(n)\le z}\frac{\Lambda_k(n)\tau_m(n)(\log n)^r}{n^{1+1/\log x}}
	\ll_m m^k c^{k+r}(k+r)!\min\{\log z,\log x\}^{k+r}
\]
and 
\als{
&\int_{-T}^T\left| \sum_{P^+(n)\le z}\frac{\Lambda_k(n)\tau_m(n)(\log n)^r}{n^{1+ 1/\log x +it}}\right|^2dt \\
	&\quad \ll_m m^{2k} c^{k+r} (k+r)!^2  \left\{ T(\log V_T)^{2k+2r} 
		+  \min\{\log z,\log x\}^{2k+2r-1}  \right\} .
}
\end{lemma}


\begin{proof} First, we show the lemma in the special case when $z=\infty$ and $m=1$. Then we show how to deduce the general case from this special one. We claim that, for any $s=\sigma+it$ with $\sigma>1$ and $t\in\SR$, we have that
\eq{lambda-e1}{
\left| \left( \frac{ \zeta^{(k)} } { \zeta}\right)^{(r)} (s) \right| 
	\le c_1^{k+r}k!r! \left( \log V_t + \frac1{|s-1|}\right)^{k+r},
}
for some absolute constant $c_1$. Observe that this estimate immediately implies both parts of the lemma in the special case when $z=\infty$ and $m=1$. So it remains to show \eqref{lambda-e1}. Lemma 4.3 in \cite{K} implies that 
\eq{lambda-e2}{
\left| \left( \frac{\zeta'}{\zeta} \right)^{(r)}(s) \right|
	\le c_2^{r+1}r!\left(\log V_t+\frac1{|s-1|}\right)^{r+1} 	\quad(r\in\SN\cup\{0\})
}
for some constant $c_2$. We will show \eqref{lambda-e1} with $c_1=4c_2$ by inducting on $k$. When $k=0$, \eqref{lambda-e1} holds trivially, whereas when $k=1$, it follows by \eqref{lambda-e2}. Assume now that \eqref{lambda-e1} holds for all $k\in\{0,1,\dots,K\}$ and all $r\in\SN\cup\{0\}$, where $K$ is some positive integer. Writing $\zeta^{(K+1)}=(\frac{\zeta'}{\zeta}\cdot \zeta)^{(K)}$, we find that 
\[
\frac{\zeta^{(K+1)}}{\zeta}(s)
	=  \sum_{j_1=0}^K  \binom {K}{j_1}   \frac{\zeta^{(K-j_1)}}{\zeta}(s)	
			\left(  \frac{\zeta'}{\zeta}  \right)^{(j_1)}(s)
\] 
and, consequently, 
\[
\left(\frac{\zeta^{(K+1)}}{\zeta}\right)^{(r)}(s)
	= \sum_{\substack{0\le j_1\le K \\ 0\le j_2 \le r } }  \binom K{j_1}\binom r{j_2}
		\left( \frac{ \zeta^{(K-j_1)} }{\zeta} \right)^{ (r - j_2 ) }(s)  \left( \frac{\zeta'}{\zeta} \right)^{( j_1 + j_2 ) }(s).
\]
By the induction hypothesis and relation \eqref{lambda-e2}, we deduce that 
\als{
\left| \left( \frac{\zeta^{(K+1)}}{\zeta} \right)^{(r)}(s) \right|
	&\le K!r! \left( \log V_t + \frac1{|s-1|} \right)^{K+r+1}
		\sum_{\substack{0\le j_1\le K\\0\le j_2\le r}}
			c_1^{K+r-j_1-j_2} c_2^{j_1+j_2+1}  \binom{j_1+j_2}{j_1} \\
	&\le K!r! \left( \log V_t + \frac1{|s-1|} \right)^{K+r+1}
		\sum_{\substack{0\le j_1\le K\\0\le j_2\le r}}
			c_1^{K+r-j_1-j_2} \left(\frac{c_1}{4} \right)^{j_1+j_2+1}  2^{j_1+j_2}\\
	&< c_1^{K+r+1} K! r!  \left( \log V_t +  \frac1{|s-1|}  \right)^{K+r+1},
}
since $c_2=c_1/4$. This completes the proof of \eqref{lambda-e1} and hence of the lemma when $z=\infty$ and $m=1$.

Finally, we show how to deduce the general case of the lemma from the case when $z=\infty$ and $m=1$. First, we prove that the same result holds with $\mu^2\Lambda_k$ in place of $\Lambda_k$. Then we deduce the lemma from this weaker statement. Indeed, if $P^+(n)\le z$ and $\mu^2(n)\Lambda_k(n)\neq0$, then $n\le z^{k}$ and $\tau_m(n)\le m^k$, since $n$ is square-free and it has at most $k$ distinct prime factors\footnote{It is well-known that $\Lambda_k$ is supported on integers with at most $k$ distinct prime factors, something which can be seen using the recursive formula $\Lambda_{k+1}=\Lambda_k\log+\Lambda*\Lambda_k$.}. 
Therefore 
\[
\frac{\tau_m(n) }{ n^{1+1/\log x} } \le \frac{ m^k } {n^{1+1/\log x}} \le \frac{ (em)^{k} }{  n^{1+1/\log(\min\{x,z\})}  }
\] 
for all such $n$. Consequently, we have that
\eq{mulambda-e1}{
\sum_{P^+(n)\le z}\frac{\mu^2(n) \Lambda_k(n)\tau_m(n)(\log n)^r}{n^{1+ 1/\log x} }
	&\le (em)^k \sum_{n=1}^\infty \frac{\Lambda_k(n)(\log n)^r}{n^{1+1/\log(\min\{x,z\})}  } \\
	&\le	m^k (ec_3)^{k+r}(k+r)!\min\{\log z,\log x\}^{k+r} ,
}
by the case with $z=\infty$ and $m=1$ we already proved. Similarly, Lemma \ref{mont} implies that
\al{
\int_{-T}^T\left| \sum_{P^+(n)\le z}\frac{\mu^2(n) \Lambda_k(n)\tau_m(n)(\log n)^r}{n^{1+ 1/\log x +it}}\right|^2dt
	&\le 3 \cdot (em)^{2k}  \int_{-T}^T   \left| \sum_{n=1}^\infty 
		\frac{\Lambda_k(n)  (\log n)^r}{n^{1+1/\log(\min\{x,z\}) +it}} \right|^2 dt \nn
	&\le m^{2k} c_4^{k+r} (k+r)!^2 \left\{ T(\log V_T)^{2k+2r}  \right. \label{mulambda-e2}   \\
	&\quad + \left.  \min\{\log z,\log x\}^{2k+2r-1}  \right\} . \nonumber	
}
The above estimates prove our claim that the lemma holds with $\mu^2\Lambda_k$ in place of $\Lambda_k$. 

We are now ready to complete the proof of the lemma. Every integer $n$ can be written uniquely as $n=ab$, where $a$ is square-free, $b$ is square-full and $(a,b)=1$. Moreover, we have that
\[
\Lambda_k(ab)	= \sum_{j=0}^k \binom kj \Lambda_{k-j}(a) \Lambda_j(b)
\]
(see, for example, \cite[p. 16]{IK}) and
\[
\log^r(ab) = \sum_{j=0}^r \binom{r}{j} (\log a)^{r-j} (\log b)^j ,
\]
Consequently,
\al{
\sum_{P^+(n)\le z} \frac{\Lambda_k(n) \tau_m(n) (\log n)^r} {n^{1+ 1/\log x +it}}
	&= \sum_{\substack{ P^+(ab)\le z,\, (a,b)=1 \\ b\ \text{square-full}} }
		\frac{\mu^2(a) \Lambda_k(ab)\tau_m(ab) \log^r(ab) }{ (ab)^{1+ 1/\log x +it } } \nn
	&= \sum_{\substack{0\le j_1\le k\\0\le j_2\le r}} \binom k{j_1} \binom r{j_2} 
		\sum_{P^+(a)\le z} \frac{\mu^2(a) \Lambda_{j_1}(a) \tau_m(a) (\log a)^{j_2}} 
			{a^{1+ 1/\log x +it}} \cdot C_{j_1,j_2}(a),   \label{lambda-e3}
}
where 
\[
C_{j_1,j_2}(a) = \sum_{\substack{P^+(b)\le z, \, (b,a)=1 \\ b\ \text{square-full} } } 
	\frac{\Lambda_{k-j_1}(b)\tau_m(b)(\log b)^{r-j_2}}{b^{1+ 1/\log x +it}}.
\]
Since 
\[
\sum_{\substack{ b\le x \\ b\ \text{square-full} } } \tau_m(b)
	\le \sum_{g^2h^3\le x} \tau_m(g^2h^3) 
	\ll_m\sqrt{x}(\log x)^{m(m+1)/2} \ll_m x^{2/3}
\]
and $\Lambda_{k-j_1}(b)\le(\log b)^{k-j_1}$, by the positivity of $\Lambda_{k-j_1}$ and the fact that $\log^{k-j_1} = 1* \Lambda_{k-j_1}$, partial summation implies that 
\eq{lambda-e4}{
|C_{j_1,j_2} (a)| \le \sum_{b\ \text{square-full}} \frac{\tau_m(b) (\log b)^{k+r-j_1-j_2}} b 
	\ll_m c_5^{k+r-j_1-j_2}(k+r-j_1-j_2)! .
}
So relation \eqref{mulambda-e1} with $j_1$ and $j_2$ in place of $k$ and $r$, respectively, and relation \eqref{lambda-e3} imply that
\als{
\sum_{P^+(n)\le z}\frac{\Lambda_k(n)\tau_m(n)(\log n)^r}{n^{1+ 1/\log x }}
	&\ll_m \sum_{\substack{0\le j_1\le k\\0\le j_2\le r}} \binom k{j_1} \binom r{j_2} 
			m^{j_1} (ec_3)^{j_1+j_2}(j_1+j_2)!  \\
	&\quad \times \min\{\log z,\log x\}^{j_1+j_2} 
	 	c_5^{k+r-j_1-j_2} (k+r-j_1-j_2)! \\ 
	&\le m^k c_6^{k+r} (k+r)!  \min\{\log z,\log x\}^{k+r} ,
}
since $(g+h)!\le 2^{g+h} g!h!$. This proves the first part of the lemma. For the second part, relation \eqref{mulambda-e2} and Lemma \ref{mont} yield that
\als{
\int_{-T}^T&\left| \sum_{P^+(n)\le z}
	\frac{ \mu^2(a) \Lambda_{j_1}(a)\tau_m(a)(\log a)^{j_2} C_{j_1,j_2}(a) }{a^{1+ 1/\log x  +it}}\right|^2dt \\
	&\ll_m c_5^{2(k+r-j_1-j_2)} (k+r-j_1-j_2)!^2 
		\int_{-T}^T   \left| \sum_{n=1}^\infty 
		\frac{  \mu^2(a) \Lambda_{j_1}(a)\tau_m(a)(\log a)^{j_2} }{a^{1+ 1/\log x  + it}} \right|^2 dt \\
	&\le c_5^{2(k+r-j_1-j_2)} (k+r-j_1-j_2)!^2  m^{2j_1} c_4^{j_1+j_2} (j_1+j_2)!^2 \\
	&\qquad\qquad \times \left\{ T(\log V_T)^{2j_1+2j_2}  +  \min\{\log z,\log x\}^{2j_1+2j_2-1}  \right\} .
}
Together with \eqref{lambda-e3} and the triangle inequality for the $L^2$-norm, the above estimate yields the second and last part of the lemma.
\end{proof}


\section{Bounds for the derivatives of $L(s,f)$}\label{l(s,f)}

In this section we estimate the partial sums of $f$ over integers with no small primes factors and deduce bounds for the derivatives of $L_y(s,f)$. In order to do so, we appeal to the so-called {\it fundamental lemma of sieve methods}. This result has appeared in the literature in many different forms (for example, see \cite[Theorem 7.2]{HR}). The version we shall use is a direct consequence of Lemma 5 in \cite{fi}.


\begin{lemma}\label{fund-lemma} Let $y\ge2$ and $D=y^u$ with $u\ge2$. There exist two arithmetic functions $\lambda^\pm:\SN\to[-1,1]$, supported in $\{d\in\SN:P^+(d)\le y,\,d\le D\}$, for which 
\[
\begin{cases}
	(\lambda^-*1)(n) =  (\lambda^+*1)(n)=1 		&\text{if}~P^-(n)>y,\cr 
	(\lambda^-*1)(n) \le 0 \le (\lambda^+*1)(n)		&\text{else}.
\end{cases}
\]
Moreover, if $g:\SN\to[0,1]$ is a multiplicative function and $\lambda\in\{\lambda^+,\lambda^-\}$, then 
\[
\sum_d \frac{ \lambda(d)g(d) }{d} =  (1+O(e^{-u}))  \prod_{p\le y} \left( 1 - \frac{ g(p) }{p} \right).
\]
\end{lemma}


Using Lemma\ \ref{fund-lemma}, we estimate the partial sums of $f$ over integers with no small primes factors. Note that the second formula in part (b) of the following lemma is similar to \cite[Lemma 4]{Pi76III}.

\begin{lemma}\label{Lbound-1} Let $f:\SN\to\SU$ be a completely multiplicative function such that 
\[
\left| \sum_{n\le x}f(n) \right| 
	\le \frac{ Mx^{\sigma_0} }{ (\log x)^A }	\quad(x\ge Q)
\] 
for some $\sigma_0\in[3/5,1]$, $Q\ge3$, $M\ge1$ and $A\ge0$. Consider $x\ge Q^4$, $y\in[10,\sqrt{x} \,]$ and $t\in\SR$.
\begin{enumerate}

\item  Let $\sigma_1=(1+\sigma_0)/2$. Then
\[
\sum_{\substack{n\le x\\P^-(n)>y}}\frac{f(n)}{n^{it}}
	\ll \frac{4^A(|t|+1)(\log y)Mx^{\sigma_1}}{(\log x)^A}+\frac{x^{1 - 1/(2\log y) }} {\log y}.
\]

\item Let $\sigma_2=(1+\sigma_1)/2=(3+\sigma_0)/4$. If $A>1$ and $y\ge V_t^{100}$, then
\[
\sum_{ \substack{ n\le x\\P^-(n)>y } } \frac{(1*f)(n)}{n^{it}}
	= \frac{x^{1-it}}{1-it} L_y(1,f)  \prod_{p\le y} \left(1-\frac1p\right) 
		+  O \left(\frac{A}{A-1} \frac{8^A(|t|+1)Mx^{\sigma_2}}{(\log x)^{A-1}} 
		+ \frac{x^{1- 1/(60\log y)} } {\log y}\right).
\]
Moreover, if $A>2$, then for $s=\sigma+it$ with $\sigma\ge\max\{\sigma_2,1-1/(120\log y)\}$ we have that\als{
\sum_{\substack{ n\le x \\ P^-(n)>y } } \frac{(1*f)(n)}{n^s}
	&=\left\{ \left( \frac1{s-1} +  \gamma_{s,y}  \right)  L_y(s,f) 
		- \frac{ x^{-s+1} }{s-1}  L_y(1,f)\right\}  \prod_{p\le y}\left(1-\frac1p\right)   \\
	& \quad +  O \left(\frac{A-1}{A-2}\frac{8^A(|t|+1)M}{(\log x)^{A-2}}+x^{- 1/(120\log y) }\right),
}
where $\gamma_{s,y}$ is defined in Lemma \ref{zeta}.
\end{enumerate}
\end{lemma}


\begin{remark}\label{Lbound-rk}
Note that 
\[
\lim_{s\to1} \left( \left( \frac1{s-1} +  \gamma_{s,y}  \right)  L_y(s,f) 
		- \frac{ x^{-s+1} }{s-1}  L_y(1,f)   \right)
		= ( \log x + \gamma_{1,y}) L_y(1,f) + L_y'(1,f).
\]
So, when $s=1$, the second formula in part (b) of Lemma \ref{Lbound-1} is interpreted to be
\als{
\sum_{\substack{ n\le x \\ P^-(n)>y } } \frac{(1*f)(n)}{n}
	&=\{ (\log x +  \gamma_{1,y})  L_y(1,f) + L_y'(1,f) \} \prod_{p\le y}\left(1-\frac1p\right)   \\
	& \quad +  O \left(\frac{A-1}{A-2}\frac{8^A(|t|+1)M}{(\log x)^{A-2}}+x^{- 1/(120\log y)}  \right).
}
\end{remark}


\begin{proof}[Proof of Lemma \ref{Lbound-1}] (a) Note that
\eq{Lbound-1-it}{
\sum_{n\le u} f(n)n^{-it} 
	&=  O(\sqrt{u}) + \int_{\sqrt{u}}^u w^{-it} d\left( \sum_{n\le w} f(n) \right) \\
	&=  O(\sqrt{u}) + u^{-it} \sum_{n\le u} f(n) + it \int_{\sqrt{u}}^u w^{-it-1}\left( \sum_{n\le w} f(n) \right) dw \\
	&\ll \sqrt{u} + \frac{ Mu^{\sigma_0} }{ (\log u)^A } + |t| \int_{\sqrt{u}}^u \frac{M w^{\sigma_0-1}}{(\log w)^A}dw 
		\ll \sqrt{u} +  \frac{2^A (|t|+1) M u^{\sigma_0} }{(\log u)^A} ,
}
for all $u\ge Q^2$. Next, we apply Lemma \ref{fund-lemma} with $D=x^{1/2}$ and we find that
\[
\sum_{ \substack{n\le x \\  P^-(n)>y } } \frac{f(n)}{ n^{it} } 
	= \sum_{n\le x} \frac{ f(n)(\lambda^+*1)(n) }{ n^{it} } 
		+ O\left( \sum_{n\le x}(\lambda^+*1-\lambda^-*1)(n) \right).
\]
For the first sum on the right hand side of the above relation, we have that
\als{
\sum_{n\le x}\frac{f(n)(\lambda^+*1)(n)}{n^{it}}
	= \sum_{\substack{d\le\sqrt{x}\\P^+(d)\le y}} \frac{\lambda^+(d)f(d)}{d^{it}}\sum_{m\le x/d}\frac{f(m)}{m^{it}} 
	&\ll \sum_{\substack{d\le\sqrt{x}\\P^+(d)\le y}} 
		\left( \frac{2^A (|t|+1) M x^{\sigma_0}}{(\log(x/d))^A d^{\sigma_0}} 
		+ \sqrt{\frac{x}{d} } \ \right)   \\
	&\ll \frac{4^A (|t|+1)(\log y) M x^{\sigma_1} }{(\log x)^A} + x^{3/4},
}
since $x/d\ge \sqrt{x} \ge Q^2$ and
\[
\sum_{\substack{ d\le \sqrt{x} \\ P^+(d)\le y}} \frac{x^{\sigma_0}}{d^{\sigma_0}}
	\le \sum_{\substack{ d\le \sqrt{x} \\ P^+(d)\le y}} \frac{x^{\sigma_1}}{d}
	\le x^{\sigma_1} \prod_{p\le y} \left(1 - \frac{1}{p} \right)^{-1}
	\ll  x^{\sigma_1}  \log y.
\]
Finally, we have that
\[
\sum_{n\le x} ( \lambda^+*1 - \lambda^-*1 )(n) 
	= \sum_{\substack{d\le\sqrt{x}\\P^+(d)\le y}} (\lambda^+(d)-\lambda^-(d) )
		\left(\frac xd +  O(1) \right) 
	\ll  \frac{x^{1- 1/(2\log y)} } {\log y},
\]
by Lemma \ref{fund-lemma}, and part (a) of the lemma follows.


\medskip

(b) Part (a) and Lemma \ref{zeta} imply that
\als{
\sum_{\substack{ n\le x \\ P^-(n)>y }} \frac{(1*f)(n)}{n^{it}}
	&= \sum_{\substack{ a\le\sqrt{x} \\ P^-(a)>y } } \frac{f(a)}{ a^{it} }
		\sum_{\substack{ b\le x/a \\ P^-(b)>y } } \frac1{b^{it}}
		+ \sum_{\substack{  b\le\sqrt{x} \\  P^-(b)>y  }}\frac1{b^{it}}
		\sum_{\substack{ \sqrt{x}< a\le x/b \\  P^-(a)>y  }}\frac{f(a)}{a^{it}}   \\
	&= \sum_{\substack{ a\le\sqrt{x} \\ P^-(a)>y } } \frac{f(a)}{ a^{it} }
		\left\{ \frac{(x/a)^{1-it}}{1-it}  \prod_{p\le y} \left(1 - \frac1p \right) 
			+ O\left(\frac{(x/a)^{1- \frac{1}{30\log y}} }{\log y} \right)  \right\} \\
	&\quad + O\left( \sum_{\substack{ b\le \sqrt{x} \\ P^-(b)>y } }
				\left\{ \frac{4^AM (|t|+1)(\log y) (x/b)^{\sigma_1}}{(\log(x/b))^A} 
					+ \frac{(x/b)^{1-\frac{1}{2\log y}}}{\log y} \right\} \right) .
}
In addition, we have that
\eq{Lbound-1-e4}{
\sum_{\substack{ b\le \sqrt{x} \\ P^-(b)> y}} \frac{x^{\sigma_1}}{b^{\sigma_1}}
	\le \sum_{\substack{ b\le \sqrt{x} \\ P^+(b)\le y}} \frac{x^{\sigma_2}}{b}
	\le x^{\sigma_2} \prod_{y<p\le \sqrt{x}} \left( 1- \frac{1}{p} \right)^{-1} 
	\ll  x^{\sigma_2} \cdot \frac{\log x}{\log y}
}
and
\al{
\sum_{\substack{ b\le \sqrt{x} \\ P^-(b)> y}} (x/b)^{1- \frac{1}{2\log y}}
	\le \sum_{\substack{ a\le \sqrt{x} \\ P^-(a)> y}} (x/a)^{1- \frac{1}{30\log y}}
	&\ll 1+ \sum_{ \log y< r\le \frac{\log x}{2} +1 }  (x/e^r)^{1-  \frac{1}{30\log y}}
		\sum_{\substack{ e^{r-1} \le a < e^r \\ P^-(a)> y}} 1 \nn
	&\ll 1+ \sum_{\log y< r\le \frac{\log x}{2} +1 }  \frac{ x^{1- \frac{1}{30\log y}} e^{\frac{r}{30\log y}}}{ \log y} \nn
	&\ll  x^{1- 1/(60\log y)} . \label{Lbound-1-e3}
}
So we deduce that
\[
\sum_{\substack{ n\le x \\ P^-(n)>y }} \frac{(1*f)(n)}{n^{it}}
	= \sum_{\substack{ a\le\sqrt{x} \\  P^-(a)>y } }  \frac{f(a)}{a}
		\frac{x^{1-it}}{1-it}   \prod_{p\le y}\left(1-\frac1p\right)
		+   O\left(  \frac{8^A(|t|+1)   Mx^{\sigma_2}  } {  (\log x)^{A-1}  }  
		+   \frac{x^{1 - 1/(60\log y) }  } {  \log y }  \right)    .
\]
Finally, note that, for any $s=\sigma+it$ with $\sigma \ge \sigma_1$, part (a) and partial summation imply that
\al{
\sum_{\substack{ a> \sqrt{x} \\  P^-(a)>y } }  \frac{f(a)}{a}
	&=  \int_{\sqrt{x}}^\infty \frac{1}{w} d\left( \sum_{\substack{ a\le w \\ P^-(a)>y}} f(a) \right) 
		=  \frac{-1}{\sqrt{x}} \sum_{\substack{ a\le \sqrt{x} \\ P^-(a)>y}} f(a) 
		+ \int_{\sqrt{x}}^\infty \left( \sum_{ \substack{ a\le w \\ P^-(a)>y}} f(n) \right)   \frac{dw }{w^2} \nn
	&\ll \frac{ 8^A (\log y) M x^{\frac{\sigma_1-1}{2}} }{ (\log x)^A } + \frac{x^{ - 1 / (4\log y) } } {\log y} \nn
	&\quad	+  \int_{\sqrt{x}}^\infty \left( \frac{4^A (\log y) M w^{\sigma_1-1} }{w (\log w)^A } 
	   	+  \frac{1}{ w^{1+ \frac{1}{2\log y} } \log y} \right) dw  \nn
	& \ll \frac{A}{A-1} \frac{8^A(\log y) M x^{\sigma_2-1} }{ (\log x)^{A-1}} +  x^{- 1/(4\log y) } ,
		\label{Lbound-1-e2}
}
since $w^{\sigma_1-1} \le x^{\frac{\sigma_1-1}{2}} = x^{\sigma_2-1}$ for $w\ge \sqrt{x}$, thus proving our first claim.


\medskip

The proof of our second claim is similar. First, note that
\eq{sievee1}{
\sum_{\substack{ n\le x \\ P^-(n)>y } } \frac{(1*f)(n)}{n^s}
	= \sum_{ \substack{ a\le\sqrt{x} \\ P^-(a)>y } }  \frac{f(a)}{a^s}  
		\sum_{\substack{b\le x/a \\ P^-(b)>y } } \frac1{b^s} 
		+\sum_{\substack{b\le\sqrt{x} \\  P^-(b)>y }}  \frac1{b^s}
	\sum_{\substack{  \sqrt{x}<a\le x/b \\   P^-(a)>y  }}   \frac{f(a)}{a^s}  .
}
Since $A\ge2$ and $\sigma\ge\sigma_2\ge\sigma_1$ by assumption, following a similar argument with the one leading to \eqref{Lbound-1-e2}, we deduce that
\eq{sievee2}{
\sum_{\substack{ u<a\le u' \\ P^-(a)>y } } \frac{f(a)}{a^s}
	&=  \int_{u}^{u'} \frac{1}{w^\sigma} d\left( \sum_{\substack{ a\le w \\ P^-(a)>y}} \frac{f(a)}{a^{it}} \right)  \\
	&\ll  \frac{ 4^A\sigma(|t|+1) (\log y)M }{ (\log u)^{A-1}   u^{\sigma-\sigma_1} }  
			+  u^{1 - \sigma - 1/(2\log y)}   \quad(u'\ge u\ge y).
}
Relations \eqref{sievee2}, \eqref{Lbound-1-e4} and \eqref{Lbound-1-e3} imply that
\als{
\sum_{\substack{b\le\sqrt{x} \\  P^-(b)>y }}  \frac1{b^s}
	\sum_{\substack{  \sqrt{x}<a\le x/b \\   P^-(a)>y  }}   \frac{f(a)}{a^s}
	& \ll \sum_{\substack{ b\le\sqrt{x} \\  P^-(b)>y }}  \frac1{b^\sigma } 
		  \left(  \frac{4^A\sigma(|t|+1)(\log y)M} { (x/b)^{\sigma-\sigma_1}  (\log(x/b))^{A-1} } 
	 	  	+  (x/b)^{1 - \sigma - 1/(2\log y) }       \right)  \\
	&\ll \frac{8^A\sigma(|t|+1)(\log y) M } { (\log x)^{A-1} x^{\sigma-\sigma_2}} 
		+ x^{1-\sigma - 1/(60\log y) }.
}
Furthermore, Lemma \ref{zeta} and relation \eqref{Lbound-1-e3} yield that
\als{
&\sum_{ \substack{ a\le\sqrt{x} \\ P^-(a)>y } }  \frac{f(a)}{a^s}  
		\sum_{\substack{b\le x/a \\ P^-(b)>y } } \frac1{b^s} \\
&\ = \sum_{\substack{ a\le\sqrt{x} \\ P^-(a)>y } } \frac{f(a)}{a^s}
		\left( \frac{1-(x/a)^{-s+1}}{s-1}  +    \gamma_{s,y}\right)   \prod_{p\le y}\left(1-\frac1p\right)    
	+   O\left(   \sum_{\substack{ a\le\sqrt{x} \\ P^-(a)>y } } \frac{1}{a^\sigma} 
		\frac{(x/a)^{1-\sigma - 1/(30\log y)} }{  \log y}  \right) \\
&\ = \sum_{\substack{ a\le\sqrt{x} \\ P^-(a)>y } } \frac{f(a)}{a^s}
		\left( \frac{1-(x/a)^{-s+1}}{s-1}  +    \gamma_{s,y}\right)   \prod_{p\le y}\left(1-\frac1p\right)    
		+ O\left( \frac{ x^{1-\sigma - 1/(60\log y)} }{\log y}   \right)
}
Inserting the above estimates into \eqref{sievee1} and using our assumption that $\sigma \ge \max\{\sigma_2,1-1/(120\log y)\}$, we deduce that
\eq{Lbound-1-e5}{
\sum_{\substack{ n\le x \\ P^-(n)>y } } \frac{(1*f)(n)}{n^s}
	&= \sum_{\substack{ a\le\sqrt{x} \\ P^-(a)>y } } \frac{f(a)}{a^s}
		\left( \frac{1-(x/a)^{-s+1}}{s-1}  +    \gamma_{s,y}\right)   \prod_{p\le y}\left(1-\frac1p\right)    \\
	&\quad + O\left( \frac{8^A(|t|+1) (\log y) M } { (\log x)^{A-1} } + \frac{1}{x^{1/(120\log y)} } \right).
}
Finally, we extend the summation over $a$ to all integers with no prime factors $\le y$ and we estimate the error term. Fix $N>x$. First, relation \eqref{sievee2} and the fact that $\gamma_{s,y}\ll \log y$, by Lemma \ref{zeta}, imply that
\eq{Lbound-1-e6}{
\sum_{\substack{ \sqrt{x} < a\le N \\ P^-(a)>y } }
	 \frac{f(a)}{a^s}   \cdot \gamma_{s,y} \prod_{p\le y}\left(1-\frac1p\right)  
	&\ll \frac{ 8^A\sigma(|t|+1) (\log y) M }{ (\log x)^{A-1}   x^{\frac{\sigma-\sigma_1}{2}} } 
		+   x^{\frac{1-\sigma}{2} - 1/(4\log y)}  \\
	&\ll \frac{ 8^A (|t|+1) (\log y) M }{ (\log x)^{A-1}  } + x^{- 1/ (8\log y) },
}
where the last inequality follows by our assumption that 
$\sigma\ge \max\{\sigma_2,1-1/(120\log y)\} \ge \max\{\sigma_1,1-1/(4\log y)\}$. Next, using the identity
\[
\frac{1-(x/a)^{-s+1}}{s-1} =  \int_1^{x/a} \frac{du}{u^s},
\]
we deduce that 
\als{
\sum_{\substack{\sqrt{x} < a\le N \\  P^-(a)>y } }  \frac{f(a)}{a^s}  \frac{1-(x/a)^{-s+1}}{s-1}
	&=  \sum_{\substack{\sqrt{x} < a \le x \\  P^-(a)>y } }  \frac{f(a)}{a^s}  \int_{1}^{x/a} \frac{du}{u^s}
		-  \sum_{\substack{x< a\le N \\  P^-(a)>y } }  \frac{f(a)}{a^s}  \int_{x/a}^{1} \frac{du}{u^s} \\
	&=  \int_1^{\sqrt{x}} \left(  \sum_{\substack{ \sqrt{x} < a \le x/u }}  \frac{f(a)}{a^s}\right) \frac{du}{u^s}
		- \int_{x/N}^1\left(\sum_{x/u<a\le N}\frac{f(a)}{a^s}\right)\frac{du}{u^s}  .
}
So relation \eqref{sievee2} yields that
\als{
\sum_{\substack{\sqrt{x} < a\le N \\  P^-(a)>y } }  \frac{f(a)}{a^s}  \frac{1-(x/a)^{-s+1}}{s-1}
	&\ll \left( \frac{8^A\sigma(|t|+1)(\log y)M } { x^{\frac{\sigma-\sigma_1}2}(\log x)^{A-1} } 
		 +   x^{\frac{1-\sigma}{2} - 1/(4\log y) } \right) \int_1^{\sqrt{x}}\frac{du}{u^\sigma}    \\
	&\quad +  \int_0^1  \left( \frac{4^A\sigma(|t|+1)(\log y)M}{(x/u)^{\sigma-\sigma_1} 
		\log^{A-1}(x/u)} +  (x/u)^{1-\sigma  - 1/(2\log y)}  \right) \frac{du}{u^\sigma}  .
}
Since $\sigma\ge \max\{\sigma_2,1-1/(120\log y)\}\ge \max\{ \sigma_1, 1- 1/(8\log y)\} $, we find that
\[
\frac{\sigma}{ x^{\frac{\sigma-\sigma_1}2} } 
	\int_1^{\sqrt{x}}\frac{du}{u^\sigma}
	= \frac{\sigma}{ x^{\frac{\sigma-\sigma_1}2} } 
		\int_0^{\frac{\log x}2}e^{v(1-\sigma)}dv
	\ll \frac{\sigma ( 1+x^{\frac{1-\sigma}2}) \log x }{ x^{\frac{\sigma-\sigma_1}{2} } } 
	\ll \log x,
\]
\[
x^{ \frac{1-\sigma}{2} -  1/(4\log y)} 
	\int_1^{\sqrt{x}} \frac{du}{u^{\sigma}}
	\ll  x^{ \frac{1-\sigma}{2}- 1/(4\log y)}  ( 1+x^{\frac{1-\sigma}2}) \log x
	\ll \frac{\log x}{x^{1/(8\log y)}} \ll \frac{\log y}{x^{1/(10\log y)}} ,
\]
\als{
\frac{\sigma}{x^{\sigma-\sigma_1}} \int_0^1\frac{du}{u^{\sigma_1}\log^{A-1}(x/u)}
	&= \frac{\sigma}{x^{\sigma-\sigma_1}}  \int_0^\infty\frac{dv}{e^{(1-\sigma_1)v}(\log x+v)^{A-1}} 
	\le  \frac{\sigma}{x^{\sigma-\sigma_1}}  \int_{0}^\infty\frac{dv}{(\log x+v)^{A-1}} \\
	&= \frac{\sigma}{x^{\sigma-\sigma_1}} \frac1{(A-2)(\log x)^{A-2}} 
	\ll \frac1{(A-2)(\log x)^{A-2}} ,
}
and 
\[
 x^{1-\sigma- 1/(2\log y) }
	\int_0^1 \frac{du}{u^{1- 1/(2\log y)}}
	\ll  x^{1-\sigma- 1/(2\log y) } \log y
	\ll \frac{\log y}{x^{1/(4\log y)}} .
\]
So we deduce that
\als{
\sum_{\substack{\sqrt{x} < a\le N \\  P^-(a)>y } }  \frac{f(a)}{a^s}  \frac{1-(x/a)^{-s+1}}{s-1}
	&\ll   \frac{A-1}{A-2}   \frac{8^A(|t|+1)(\log y)M } {  (\log x)^{A-2} } + \frac{\log y}{x^{1/(10\log y)}} .
}
Inserting the above estimate and relation \eqref{Lbound-1-e6} into \eqref{Lbound-1-e5}, we complete the proof of the second portion of part (b) of the lemma too.
\end{proof}


Finally, we have the following pointwise bound on the derivatives of $L_y(s,f)$.


\begin{lemma}\label{Lbound-2} Let $Q\ge 10$, $\epsilon\in(0,1)$, $y\ge Q^\epsilon+1$ and $f:\SN\to\SC$ be a completely multiplicative function.

\begin{enumerate}

\item If $f$ satisfies \eqref{small-Q} for some $A\ge2$, then for $\sigma\ge1$ and $k\in\SZ\cap[0,A-1]$ we have that
\[
| L_y^{(k)}(\sigma,f) | 
	\ll_{A,\epsilon} (\log y)^k  \min \left\{  \frac{A-k}{A-k-1},  
		\log\left(2+ \frac1{(\sigma-1)\log y} \right) \right\} .
\]

\item If $f$ satisfies \eqref{small-power} with $\delta=1/\log Q$, then for $\sigma\ge \max\{1-1/(4\log Q),1-1/(4\log y)\}$ and $k\in\SN\cup\{0\}$, we have that
\[
| L_y^{(k)}(\sigma,f) | \ll \frac{(k+1)! (4\log(yQ))^k}{\epsilon}.
\]
\end{enumerate}
\end{lemma}


\begin{proof} We are going to prove the two parts simultaneously. We assume that $f$ satisfies the relation
\eq{small-alt}{
\left| \sum_{n\le x} f(n) \right| \le \frac{x^{1-\alpha/\log Q} (\log Q)^{B-2}}{(\log x)^B} \quad(x\ge Q)
}
for some $B\ge2$ and some $\alpha\in[0,1]$. In part (a) we know that this is true with $B=A$ and $\alpha=0$, whereas in part (b) relation \eqref{small-alt} holds with $B=2$ and $\alpha=1$.  Furthermore, we consider $k\in\SN\cup\{0\}$ and $\sigma\ge\max\{1-\alpha/(4\log Q),1-1/(4\log y)\}$. Lemma \ref{Lbound-1}(a) and partial summation imply that
\als{
L_y^{(k)}(\sigma,f) 
	&= \int_y^\infty\frac{(-\log x)^k}{x^\sigma} d\left( \sum_{\substack{n\le x\\ P^-(n)>y}}f(n)\right) \\
	&= (-1)^k \int_y^\infty  \frac{ \sigma(\log x)^k  -  k(\log x)^{k-1} }{x^{\sigma+1}} 
		\left( \sum_{\substack{n\le x\\ P^-(n)>y}}f(n)\right) dx \\
	&\ll  (k+\sigma)
		 \int_{\max\{y^2,Q^4\}}^\infty \frac{(\log x)^k}{x^{\sigma+1}} 
		 	\left(  \frac{(\log y)(4\log Q)^{B-2} x^{1- \alpha/(2\log Q) } }  {(\log x)^{B} } 
			+  \frac{x^{1- 1/(2\log y)}} {\log y }  \right) dx \\
	&\quad +  \int_{ y} ^{\max\{y^2,Q^4\}} \frac{(k+\sigma)(\log x)^k}{x^{\sigma} \log y} dx.
}
The last integral appearing in the above relation is at most
\als{
 \int_{ y} ^{\max\{y,Q^4\}} \frac{(k+\sigma)(4\log(yQ))^k}{x^{\sigma} \log y} dx
 	&\ll  \frac{(k+1) (4\log(yQ))^{k+1} }{\log y} 
		\ll \frac{(k+1) (4\log(yQ))^{k}}{\epsilon},
}
by our assumption that $y\ge Q^\epsilon+1$. Thus making the change of variable $x=y^u$ yields that
\al{
\frac{ L_y^{(k)}(\sigma,f) }{k+1} 
	&\ll  \sigma
		 \int_1^\infty\left(  \frac{ (\log y)^{k-B+2} (4\log Q)^{B-2}}
		{u^{B-k} y^{(\sigma-1+\alpha/(2\log Q))u}  } 
			+ \frac{(\log y)^k u^k} {y^{(\sigma -1)u} e^{u/2} }  \right) du 
			+ \frac{ (4\log(yQ))^{k}}{\epsilon} \nn
	&\ll  \int_1^\infty\left(  \frac{ \sigma 4^B \epsilon^{2-B}(\log y)^k }
		{  u^{B-k} y^{(\sigma-1+\alpha/(2\log Q))u} }
			+ \frac{(\log y)^k u^k}{ e^{u/4} }  \right) du  +  \frac{(4\log(yQ))^{k}}{\epsilon} \nn
	&\le  \sigma 4^B \epsilon^{2-B}(\log y)^k
		\int_1^\infty   \frac{ du } {  u^{B-k } y^{(\sigma-1+\alpha/(2\log Q))u} }
			+  4^{k+1} k! (\log y)^k  +  \frac{(4\log(y Q) )^{k}}{\epsilon}  , \label{Lbound-2-e1}
}
since $\sigma \ge 1-1/(4\log y)$ and $\int_0^\infty t^k e^{-t} dt = k!$.

It remains to estimate the integral in \eqref{Lbound-2-e1}. If $\alpha=1$, $B=2$ and $\sigma\ge\max\{1-1/(4\log Q),1-1/(4\log y)\}$, then
\als{
	 \int_1^\infty   \frac{\sigma \cdot (\log y)^k  du } {  u^{B-k } y^{(\sigma-1+\alpha/(2\log Q))u} }
		\ll (\log y)^k \int_1^\infty   \frac{ u^{k} du } { y^{u/(4\log Q)} }
		\le \frac{ k! (4\log Q)^{k+1} }{\log y}  
		\ll \frac{ k! (4\log Q)^k }{\epsilon},
}
by our assumption that $y\ge Q^\epsilon+1$. The above inequality together with \eqref{Lbound-2-e1} completes the proof of part (b). Finally, in order to show part (a), we estimate the  integral in \eqref{Lbound-2-e1} when $B=A$, $\alpha=0$, $k\le A-1$ and $\sigma\ge1$. In this case we have that
\eq{Lbound-2-e2}{
\sigma \int_1^\infty   \frac{ du } {  y^{(\sigma-1+\alpha/(2\log Q)) u } u^{B-k} }
	\ll \int_1^\infty   \frac{ du } {u^{A-k}} 
	= \frac{1}{A-k-1}
}
and, if we set $Y=\max\{1,(  (\sigma-1)\log y  )^{-1}    \}$, then
\eq{Lbound-2-e3}{
\sigma \int_1^\infty   \frac{ du } {  y^{(\sigma-1+\alpha/(2\log Q)) u } u^{B-k} }
	\ll \int_1^\infty   \frac{ du } {  y^{(\sigma-1)u/2} u} 
	&\le \int_1^Y  \frac{ du } {   u } +  \frac{1}{Y} \int_Y^\infty   \frac{ du } {  y^{ (\sigma - 1) u /2 }}  \\
	&\le \log Y + 2e^{-1/2} .
}
Combining relations \eqref{Lbound-2-e2} and \eqref{Lbound-2-e3} with \eqref{Lbound-2-e1} completes the proof of part (a) of the lemma too.
\end{proof}


\section{Distances of multiplicative functions}\label{distances}

This section is devoted to studying some properties of the distance function and establishing Theorem \ref{real-complex}. We start with two straightforward results, which we state for easy reference. The first one of them links the size of $L_y(s,f)$ to a certain distance function. In its current formulation it is due to Granville and Soundararajan\ \cite{gs}, but similar results have appeared before in work of Elliott \cite[Lemma 6.6 in p. 230, and p. 253]{E} and Tenenbaum\ \cite[Lemma 2.1, p. 326]{T}. A sketch of its proof is given in \cite[Lemma 3.2]{K}.


\begin{lemma}\label{dist-l1}Let $x,y\ge2$, $t\in\SR$ and $f:\SN\to\SU$ a multiplicative function. Then 
\[
\log\left|L_y\left(1+\frac1{\log x}+it,f\right)\right|=\sum_{y<p\le x}\frac{\Re( f(p) p^{-it} ) }{p} +O(1).
\]
\end{lemma}


The second lemma is an immediate consequence of Lemmas \ref{Lbound-2}(a) and \ref{dist-l1}.


\begin{lemma}\label{dist-l2} Let $Q\ge3$ and $f:\SN\to\SU$ be a completely multiplicative function satisfying \eqref{small-Q} with $A=2$. Then there is an absolute constant $c$ such that
\[
\sum_{y<p\le x}\frac{\Re(f(p))}p\le c    \quad (x\ge y \ge Q). 
\]
\end{lemma}


Next, we have the triangle inequality \cite{gs} for the distance function defined in Section \ref{intro}.

\begin{lemma}\label{triangle}Let $f,g,h:\SN\to\SU$ be multiplicative functions and $x\ge y\ge1$. Then 
\[
\SD(f,g;y,x) +  \SD(g,h;y,x) \ge   \SD(f,h;y,x).
\]
\end{lemma}


The following result is Lemma 3.3 in \cite{K}.

\begin{lemma}\label{dist1} Let $y_2\ge y_1\ge y_0\ge2$. Consider a multiplicative function $f:\SN\to\SU$ such that 
\[
\left| L_{y_0}'\left( 1+ \frac1{\log x},f \right) \right|  \le c\log y_0   \quad(y_1\le x\le y_2)
\]
for some $c\ge1$ and 
\[
\SD^2(f(n),\mu(n);y_0,x)  \ge  \delta  \log\left( \frac{\log x}{\log y_0} \right) -  M  \quad(y_1\le x\le y_2)
\] 
for some $\delta>0$ and $M\ge0$. Then 
\[
\left| \sum_{y_1<p\le y_2} \frac{f(p)}p \right|  \ll_{c,M}\frac1\delta.
\]
\end{lemma}


We are now in position to show Theorem \ref{real-complex}.

\begin{proof}[Proof of Theorem \ref{real-complex}] (a) Assume that $f^2$ satisfies \eqref{small-Q} with $A=2+\epsilon$. Then Lemmas \ref{dist-l2} and \ref{triangle} imply that 
\als{
4\cdot \SD^2(f(n),\mu(n);y,x)
	&= \left(\SD(f(n),\mu(n);y,x)+\SD(\mu(n),\overline{f(n)};y,x)\right)^2  \\
	& \ge \SD^2(f(n),\overline{f(n)};y,x)
	= \SD^2(f^2(n),1;y,x)
	\ge \log\left(\frac{\log x}{\log y}\right) +  O(1) ,
}
for all $x\ge y\ge Q$. So part (a) of the theorem follows by Lemma \ref{dist1}, which is applicable by Lemma \ref{Lbound-2}(a).

\medskip

(b) Assume that $f$ assumes values in $[-1,1]$ and consider $t\in\SR$ with $Q\ge\max\left\{V_{2t}^{100} ,e^{1/|t|}\right\}$. Applying the second part of Lemma \ref{zeta} with $s=1+1/\log x+2it$ implies that 
\[
\zeta_y(s)
	= \lim_{u\to\infty} \sum_{\substack{ n\le u \\ P^-(n)>y}} \frac{1}{n^s} 
	=\left( \frac{1}{ s-1} + \gamma_{s,y}\right) \prod_{p\le y}\left(1-\frac1p\right)
	\ll 1,
\]
for every $y\ge Q$, since $|t|\ge1/\log y$ and $\gamma_{s,y}\ll\log y$. Therefore Lemmas \ref{dist-l1} and \ref{triangle} give us that
\als{
4 \cdot \SD^2(f(n),\mu(n)n^{it};y,x) 
	&=  \left(\SD(\mu(n)n^{-it},f(n);y,x)  +   \SD(f(n),\mu(n)n^{it};y,x)\right)^2   \\
	&    \ge  \SD^2(\mu(n)n^{-it},\mu(n)n^{it};y,x)  =   \SD^2(1,n^{2it};y,x)    \\
	&=  \log\left( \frac{\log x}{\log y} \right) + O(1)  -   \log\left|\zeta_y\left(1+\frac1{\log x} + 2it \right)\right|  \\
	& \ge  \log\left(\frac{\log x}{\log y}\right) +  O(1)  ,
}
for all $x\ge y\ge Q$. So, applying Lemma \ref{dist1}, which is possible by Lemma \ref{Lbound-2}(a), we deduce that
\eq{thm2e1}{
\sum_{y<p\le x} \frac{f(p)}{p^{1+it}} \ll_\epsilon  1  \quad(x\ge y\ge Q),
}
which completes the proof in this case.

Finally, assume that $|t|\le1/\log Q$ and $Q\ge V_{2t}^{100}$. Let $x\ge Q$ and $z=\min\{x,e^{1/|t|}\}\ge Q$. We claim that
\eq{thm2e1-2}{
\sum_{z<p\le x} \frac{f(p)}{p^{1+it}} \ll_\epsilon  1  .
}
Indeed, if $z\ge e^{1/|t|}$, then \eqref{thm2e1-2} follows by the argument leading to \eqref{thm2e1} with $z$ in place of $y$ and $e^{1/|t|}$ in place of $Q$, whereas if $z<e^{1/|t|}$, then $z=x$ and hence \eqref{thm2e1-2} holds trivially. Consequently,
\als{
\sum_{Q<p\le x}\frac{\Re(f(p)p^{-it})}p
	= \sum_{Q<p\le z} \frac{ \Re(f(p)^{-it})}{p} + O_\epsilon(1)
	&= \sum_{Q<p\le z} \frac{ f(p)+O(|t|\log p) }{p} + O_\epsilon(1) \\
	&= \sum_{Q<p\le z} \frac{f(p)}p + O_\epsilon(1). 
}
So, for every $u\ge x\ge z$, Lemmas \ref{dist-l1} and \ref{dist-l2} yield that 
\als{
\sum_{Q<p\le x} \frac{\Re(f(p)p^{-it})}{p} 
	&=  \sum_{Q<p\le z}\frac{f(p)}{p} +   O_\epsilon(1)
	\ge  \sum_{Q<p\le x}\frac{f(p)}{p}  + O_\epsilon(1) \\
	&\ge  \sum_{Q<p\le u}\frac{f(p)}{p}  + O_\epsilon(1)
	=  \log L_Q\left(1+\frac1{\log u},f\right) + O_\epsilon(1).
}
Letting $u\to\infty$ completes the proof of the theorem in this last case too.
\end{proof}


Finally, we need the following result which is a strengthening of Lemma \ref{dist1} when $y_0=y_1$.


\begin{lemma}\label{dist-l3} Let $y_1\ge y_0\ge2$ and let $f:\SN\to\SU$ be a multiplicative function such that 
\[
\left| L_{y_0}'\left( 1 + \frac1{\log x}, f \right)\right| \le c_1\log y_0  \quad(x\ge y_0)\]
and 
\[
\SD^2(f(n),\mu(n);y_0,x) \ge  \delta  \log\left( \frac{\log x}{\log y_0} \right)  -   M    \quad(y_0\le x\le y_1)
\] 
for some $c_1>0$, $\delta>0$ and $M\ge0$. There is a constant $c_2$, depending at most on $c_1$, such that if $y_1\ge y_0^{\exp\{c_2M/\delta\}}$, then for all $z \ge y_0':= y_0^{\exp\{2M/\delta\}}$ we have that
\[
\sum_{y_0'<p\le z}\frac{f(p)}p  \ll_{c_1}   \frac{1}{\delta}   .
\] 
\end{lemma}


\begin{proof} Let $y_1\ge y_0^{\exp\{c_2M/\delta\}}$. Note that 
\eq{diste1}{
\SD^2(f(n),\mu(n);y_0,x)
	\ge  \frac{\delta}{2}  \log\left( \frac{\log x}{\log y_0} \right)
	\ge c_3\delta \sum_{y_0<p\le x} \frac1p   \quad(y_0'\le x\le y_1) ,
}
for some positive constant $c_3\le1/2$. For $y\ge y_1\ge y_0'$ set 
\[
\epsilon(y) =  \min_{y_0'\le x\le y}\frac{\SD^2(f(n),\mu(n);y_0,x)}{\sum_{y_0<p\le x}1/p}
\]
and note that 
\eq{diste6}{
\left| \sum_{y_0'<p\le x}  \frac{f(p)}{p} \right| \ll_{c_1}\frac1{\epsilon(y)}   \quad(y_0'\le x\le y),
} 
by Lemma \ref{dist1}. We claim that 
\eq{diste2}{
\epsilon(y) \ge c_3 \delta.
}
Assume on the contrary that $\epsilon(y)<c_3\delta$. Let $x_0\in[y_0',y]$ be such that 
\[
\SD^2(f(n),\mu(n);y_0,x_0) =  \epsilon(y)  \sum_{y_0<p\le x_0}\frac1p.
\] 
We must have that $x_0>y_1$; otherwise, \eqref{diste1} would contradict our assumption that $\epsilon(y)<c_3\delta$. Moreover, we have that
\eq{diste5}{
\SD^2(f(n),\mu(n);y_0,x_0)
	&=  \epsilon(y)\sum_{y_0<p\le x_0} \frac1p 
		=  \frac{\epsilon(y)}{1-\epsilon(y)}\sum_{y_0<p\le x_0}\frac{-\Re(f(p))}p   \\
	&=  O\left(  \frac{\epsilon(y)M} {\delta} \right)  
	+   \frac{\epsilon(y)}{1-\epsilon(y)}   \sum_{y_0'<p\le x_0}\frac{-\Re(f(p))}p
	\ll_{c_1}M,
}
by \eqref{diste6} and our assumption that $\epsilon(y)<c_3\delta$. On the other hand, we have that 
\[
\SD^2(f(n),\mu(n);y_0,x_0) 
	\ge \SD^2(f(n),\mu(n);y_0,y_1)
	\ge \frac{\delta}{2}  \log\left( \frac{\log y_1}{\log y_0} \right)
	\ge \frac{c_2M}2,
\]
by \eqref{diste1}. If $c_2$ is large enough, the above inequality contradicts \eqref{diste5}. This implies that relation \eqref{diste2} does indeed hold. Combining relations \eqref{diste6} and \eqref{diste2}, we deduce that 
\[
\sum_{y_0'<p\le x}\frac{f(p)}p  \ll_{c_1}  \frac{1}{\delta} \quad(y_0'\le x\le y).
\] 
Since the above inequality is true for all $y\ge y_1$, the desired result follows.
\end{proof}


\section{Proof of Theorems \ref{complex-dist} and \ref{min-dist} }\label{dist-proofs}

In this section we prove Theorems \ref{complex-dist} and \ref{min-dist}. We start by proving a weaker version of the former.

\begin{theorem}\label{complex-dist-2} Let $\epsilon>0$, $Q\ge3$ and $f:\SN\to\SU$ be a completely multiplicative function that satisfies \eqref{small-Q} with $A=2+\epsilon$. Then there is some $Q'\in[Q,+\infty]$ such that 
\[
\SD^2(f(n),\mu(n); Q, Q') \ll_\epsilon1
	\quad\text{and}\quad
\sum_{Q'<p\le z}\frac{f(p)}p  \ll_\epsilon1	\quad(z\ge Q').
\]
\end{theorem}


\begin{proof} Let $c$ be a large fixed integer that depends at most on $\epsilon$, to be chosen later, and assume, without loss of generality, that $Q\ge c$. Define $Y$ to be the smallest integer $y\ge Q$ such that 
\eq{complex-dist-e1}{
\SD^2(f(n),\mu(n);y,y^c)  \le   \frac1{\log^2c},
}
if such an integer exists; else, set $Y=\infty$. This definition immediately implies that 
\eq{complex-dist-e2}{
\SD^2(f(n),\mu(n);Q,z)
	\ge  \frac{c_1}{(\log c)^3}  \log\left( \frac{\log z}{\log Q} \right) -  c_2\log c
} 
for $Q\le z<Y$, where $c_1$ and $c_2$ are some appropriate absolute constants. By the above relation and Lemma \ref{dist-l3}, which is applicable by Lemma \ref{Lbound-2}(a), there is a constant $c_3=c_3(\epsilon)$, independent of our choice of $c$, such that if $Y\ge Q_1=Q^{\exp\{c_3 (\log c)^4 \}}$, then 
\[
\sum_{Q_0<p\le z}\frac{f(p)}p   \ll_\epsilon ( \log c)^3  
\quad\text{for all}\quad
z>Q_0 := Q^{\exp\{2c_2(\log c)^4/c_1\}} .
\] 
So the theorem follows in this case by taking $Q'=Q_0=Q^{O_{c}(1)}$. Consequently, we may assume that $Y<Q_1$. 

Now, let $y\ge Q$ that satisfies \eqref{complex-dist-e1}. We are going to show that it is possible to control the size of $L_y(1,f)$ very well. We start by observing that
\[
\sum_{y<p\le y^c}\frac{|1+f(p)|}p
	\le  \sum_{y<p\le y^c}  \frac{\sqrt{2(1+\Re( f(p) ) ) } }p
	\le\SD(f(n),\mu(n);y,y^c)  \left(\sum_{y<p\le y^c}  \frac2p\right)^{1/2}   \ll\frac1{\sqrt{\log c}},
\]
by the Cauchy-Schwarz inequality and the inequality $|1+u|^2\le2\Re(1+u)$ for $u\in\SU$. Therefore
\als{
\sum_{ \substack{1<n\le y^c \\ P^-(n)>y } } \frac{|1*f|(n)}{n}
	&\le -1 +  \prod_{y<p\le y^c} \left( 1 + \frac{|1*f|(p)}{p}  +   \frac{|1*f|(p^2)}{p^2}  +   \cdots   \right)  \\
	& =  -1 +  \exp\left\{\sum_{y<p\le y^c}  \frac{ |1+f(p)| }{p}
		+  O\left(\frac1y\right)\right\} \\
	&= -1 +  \exp\left\{  O\left( \frac{1}{\sqrt{\log c}} + \frac1y   \right)\right\} 
		\ll \frac1{\sqrt{\log c}},
}
since $y\ge Q\ge c$. On the other hand, Lemma \ref{Lbound-1}(b) yields that
\als{
\sum_{\substack{n\le y^c \\  P^-(n)>y }}  \frac{(1*f)(n)}{n^\sigma}
	&=  \left\{ \left(  \frac1{\sigma-1}  +   \gamma_{\sigma,y}\right)  L_y(\sigma,f)  
		-   \frac{y^{-c(\sigma-1) } }{\sigma-1}  L_y(1,f)  \right\}  \prod_{p\le y} \left( 1 - \frac1p  \right)
	+  O_\epsilon(c^{-\epsilon}),
}
uniformly for $\sigma>1$ and $y\ge Q$. So we deduce that 
\eq{complex-dist-e3}{
\left\{ \left( \frac1{\sigma-1} +   \gamma_{\sigma,y}  \right)  L_y(\sigma,f)
	-  \frac{y^{-c(\sigma-1) } }{\sigma-1}  L_y(1,f)  \right\}  \prod_{p\le y}\left( 1 - \frac1p \right)
	= 1 + O_\epsilon\left( \frac1{\sqrt{\log c}} \right).
}

In order to proceed further, we distinguish two cases, according to whether $L(1,f)$ vanishes or not. First, assume that $L(1,f)=0$, in which case we also have that $L_y(1,f)=0$. Note that letting $x\to\infty$ in the second formula of Lemma \ref{zeta} yields the identity 
\[
\zeta_y(\sigma) =  \left(\frac1{\sigma-1} +  \gamma_{\sigma,y}\right)  \prod_{p\le y}\left(1-\frac1p\right).
\]
So \eqref{complex-dist-e3} becomes 
\[
\zeta_y(\sigma)L_y(\sigma,f) =  1 +  O_\epsilon\left(\frac1{\sqrt{\log c}}\right).
\]
In the above relation we choose $c=c(\epsilon)$ large enough and we set $\sigma=1+1/\log x$ with $x\ge y$. Then Lemma \ref{dist-l1} implies that 
\[
\SD^2(f(n),\mu(n);y,x) 
	= \log \left| \zeta_y\left(1+\frac{1}{\log x}\right) L_y\left(1+\frac{1}{\log x},f\right) \right|  +O(1) \ll1.
\]
Since this holds for all $x\ge y$, selecting $y=Y=Q^{O_{c,\epsilon} (1)}$ completes the proof of the theorem in this case by taking $Q'=\infty$.

Lastly, we consider the case when $L(1,f)\neq0$. As in Remark \ref{Lbound-rk}, letting $\sigma\to1^+$ in \eqref{complex-dist-e3} yields that
\[
\left\{ \left( c\log y +   \gamma_{1,y}  \right)  L_y(1,f)
	 + L_y'(\sigma,f)  \right\}  \prod_{p\le y}\left( 1 - \frac1p \right)
	= 1 + O_\epsilon\left( \frac1{\sqrt{\log c}} \right) .
\] 
Dividing the above formula by 
\[
P(y) 	=  L_y(1,f)  \prod_{p\le y} \left( 1 - \frac1p  \right)
	=L(1,f)\prod_{p\le y} \left(1-\frac{f(p)}{p}  \right)  \left(1-\frac1p\right)
\]
gives us that
\eq{complex-dist-e4}{
c\log y +  \gamma_{1,y}  +  \frac{L_y'}{L_y}(1,f)  =    \frac{1 +   O_\epsilon(\log^{-1/2}c)  } {  P(y)  }  .
}
Since $\gamma_{1,y}\ll\log y$ by Lemma \ref{zeta} and $\frac{L_y'}{L_y}(1,f)=\frac{L'}{L}(1,f)+O(\log y)$, relation \eqref{complex-dist-e4} becomes
\eq{complex-dist-e5}{
c\log y =  -\frac{L'}{L}(1,f) +   \frac{1+O_\epsilon(\log^{-1/2}c) } { P(y) }  +   O(\log y).
}

We are going to use the above formula to estimate the size of $L_y(1,f)$. The problem is that we do not know how big $(L'/L)(1,f)$ is. However, if $y_2\ge y_1^2$ and both $y_1$ and $y_2$ satisfy \eqref{complex-dist-e1}, then relation \eqref{complex-dist-e5} is true with $y=y_1$ and $y=y_2$. Subtracting the first one of these formulas from the second one yields the estimate
\eq{complex-dist-e6}{
\frac1{P(y_2)} -  \frac1{P(y_1)} 
	+  O_\epsilon\left(\frac1{\sqrt{\log c}}\left(  \frac1{P(y_1)}  +\frac1{P(y_2)}  \right) \right) 
=  c  \log\frac{y_2}{y_1} +   O(\log y_2) \asymp c\log y_2,
}
provided that $c$ is large enough. Note that 
\[
\left|  \frac{P(y_1)}{P(y_2)}  \right|  =  \exp\left\{  O\left(\frac1{y_1}  \right) +   \SD^2(f(n),\mu(n);y_1,y_2)  \right\}.
\]
Therefore, if $\SD(f(n),\mu(n);y_1,y_2)\ge1$ and $c$ is large enough, then $1/|P(y_2)|\ge e^{1/2} /|P(y_1)| $. Together with \eqref{complex-dist-e6}, this implies that 
\[
\frac1{|P(y_2)|} \asymp c \log y_2,
\]
that is to say, $|L_{y_2}(1,f)|\asymp1/c$. Combining this last relation with \eqref{complex-dist-e1} for $y=y_2$, we find that 
\eq{complex-dist-e7}{
|L_{y_2^c}(1,f)|  \asymp  |L_{y_2}(1,f)|  
	 \exp\left\{ -  \sum_{y_2<p\le y_2^c}  \frac{\Re( f(p) ) }p  \right\} 
	 \asymp  \frac{1}{c} \exp\{\log c\}=1.
}

Having proven this, it is relatively easy to complete the proof of the theorem. First, define a sequence $Y_1,Y_2,\dots$ inductively as follows. Set $Y_1=Y$ and let $Y_{j+1}$ be the smallest integer $y\ge Y_j^c$ which satisfies \eqref{complex-dist-e1}, provided that such an integer exists. Let $J$ be the total number of elements in this sequence (for now we allow the possibility that $J=\infty$, even though we will show that this is impossible under the assumption that $L(1,f)\neq0$). If $J\le2$, then \eqref{complex-dist-e2} holds for all $z\ge Q$, possibly with a different constant $c_2'$ in place of $c_2$, and Lemma \ref{dist-l1}, which is applicable by Lemma \ref{Lbound-2}(a), completes the proof of the theorem with $Q'=Q$. So assume that $J>2$ and consider an integer $j\in[1,J-1]$. Then, for $u>Y^c_{j+1}$, Lemmas \ref{dist-l1} and \ref{dist-l2}, and relation \eqref{complex-dist-e1} with $y=Y_{j+1}$, imply that
\als{
\left|L_{Y_j^c}  \left( 1 +  \frac1{\log u},  f \right)\right|
	&\asymp \exp \left\{ \sum_{Y_j^c< p\le Y_{j+1} } \frac{\Re(f(p))}{p} 
		+ \sum_{Y_{j+1}< p\le Y_{j+1}^c} \frac{\Re(f(p))}{p} 
		+ \sum_{Y_{j+1}^c< p\le u} \frac{\Re(f(p))}{p} \right\} \\
	& \ll  \exp\left\{\sum_{Y_{j+1}<p\le Y_{j+1}^c}  \frac{\Re( f(p)) }{p}   \right\}
		\ll \frac1c .
}
Consequently, 
\eq{complex-dist-e8}{
|L_{Y_{j}^C}(1,f)| 
	=  \lim_{u\to\infty} \left|  L_{Y_{j}^C} \left(1 +  \frac1{\log u}, f \right) \right|
	\ll \frac1c.
}
This implies that 
\eq{complex-dist-e9}{
\SD(f(n),\mu(n);Y_1,Y_j )<1;
} 
otherwise, \eqref{complex-dist-e7} with $y_2=Y_j$ would yield the estimate $|L_{Y_j^c}(1,f)|\asymp1$, which contradicts \eqref{complex-dist-e8}, provided that $c$ is large enough. Together with our assumption that $Y\le Q_0 = Q^{O_{\epsilon,c}(1) }$, relation \eqref{complex-dist-e9} yields that
\[
\SD(f(n),\mu(n); Q ,Y_{J-1}) \ll_{\epsilon,c} 1.
\]
In particular, $J<\infty$; otherwise, letting $j\to\infty$ in \eqref{complex-dist-e9} would imply that $L_Q(1,f)=0$, a contradiction to our assumption that $L(1,f)=0$. We claim that the theorem holds with $Q'=Y_{J-1}$. Indeed, its first assertion is an immediate consequence of \eqref{complex-dist-e9}. For the second one, note that
\[
\SD^2(f(n),\mu(n);Y_{J-1},z)
	\ge  \frac{c_4}{ (\log c)^3 } \log\left(\frac{\log z}{\log Y_{J-1}}\right) -   c_5 \log c    \quad(z\ge Y_{J-1}),
\] 
by the definition of $Y_{J-1}$. Hence Lemma \ref{dist1}, which is applicable by Lemma \ref{Lbound-2}(a), implies that 
\eq{complex-dist-e10}{
\sum_{Y_{J-1}<p\le z} \frac{f(p)}{p} \ll_c  1   \quad(z\ge Y_{J-1}).
} 
So the second assertion of the theorem holds too, thus completing the proof in this last case.
\end{proof}


We are now ready to prove Theorem\ \ref{complex-dist}.

\begin{proof}[Proof of Theorem \ref{complex-dist}] We will prove the theorem with $Q'$ as in Theorem \ref{complex-dist-2}. First, note that an immediate consequence of Lemma \ref{dist-l1} and Theorem \ref{complex-dist-2} is that 
\eq{complex-dist-e11}{
|L_y(1,f)|\asymp_\epsilon \frac{\log y}{\log(yQ')} \quad(y\ge Q),
}
which shows the last part of the theorem. It remains to prove that
\[
\sum_{Q<p\le Q'} \frac{|1+f(p)|}{p} \ll_\epsilon 1.
\]
Observe that the function $g=1*f*1*\overline{f}$ assumes non-negative real values and that $g(p)=2+2\Re(f(p))$. Moreover,
\als{
\sum_{\substack{ n\le x \\ P^-(n)>Q}}g(n) 
	&= \sum_{\substack{ a\le\sqrt{x} \\ P^-(a)>Q }} (1*f)(a) 
		\sum_{\substack{ b\le x/a \\ P^-(b)>Q}}(1*\overline{f})(b) 
		+ \sum_{ \substack{ b\le\sqrt{x} \\ P^-(b)>Q} }(1*\overline{f})(b) 
			\sum_{\substack{ a\le x/b \\ P^-(a)> Q }}(1*f)(a) \\
	&\qquad - \left( \sum_{\substack{ a\le\sqrt{x} \\ P^-(a)>Q }} (1*f)(a) \right)
		\left(\sum_{\substack{ b\le\sqrt{x} \\ P^-(b)>Q }} (1*\overline{f})(b) \right).
}
So
\als{
0&\le \sum_{Q<p\le x}(1+\Re(f(p))) 
	\le \frac{1}{2}\sum_{ Q<p\le x} g(p)
		\le \frac{1}{2}\sum_{\substack{ n\le x \\ P^-(n)>Q}}g(n) \\
	&= \Re\left(\sum_{\substack{ a\le\sqrt{x} \\ P^-(a)>Q }} (1*f)(a) 
		\sum_{\substack{ b\le x/a \\ P^-(b)>Q}}(1*\overline{f})(b) \right)
		- \frac{1}{2} \left| \sum_{\substack{ a\le\sqrt{x} \\ P^-(a)>Q }} (1*f)(a) \right|^2.
}
Now, Lemma\ \ref{Lbound-1}(b) implies that
\[
\sum_{\substack{ n\le u\\ P^-(n)>Q }}(1*f)(n) 
	= u L_Q(1,f)\prod_{p\le Q}\left(1- \frac1p \right)
		+ O\left(\frac{u(\log Q)^{\epsilon}}{(\log u)^{1+\epsilon}} \right)  \quad(u \ge Q^4).
\]
Therefore, for every $x\ge Q^8$ we have that
\al{
\frac{1}{x}\sum_{Q<p\le x}(1+\Re(f(p))) 
	&\le \Re\left(
	 	\sum_{\substack{ a\le\sqrt{x} \\ P^-(n)>Q }} \frac{(1*f)(a)}{a}
		\left\{ \overline{L_Q}(1,f)\prod_{p\le Q}\left(1- \frac1p \right)
		+ O\left(\frac{(\log Q)^{\epsilon}}{(\log x)^{1+\epsilon}} \right)  \right\} 
			\right) \nn
	&\quad - \frac{1}{2}  \left|  L_Q(1,f)\prod_{p\le Q}\left(1- \frac1p \right)
		+ O\left(\frac{(\log Q)^{\epsilon}}{(\log x)^{1+\epsilon}} \right)  \right|^2  .  \label{complex-dist-e12} 
}
We shall estimate each of the terms appearing above separately. First, the second formula of Lemma\ \ref{Lbound-1}(b) with $s=1$ (see also Remark \ref{Lbound-rk}) implies that
\als{
\sum_{\substack{ a\le\sqrt{x} \\ P^-(n)>Q } } \frac{(1*f)(a)}{a}
	&= \left\{ (\log\sqrt{x} + \gamma_{1,Q} ) L_Q(1,f) + L_Q'(1,f) \right\} \prod_{p\le Q} \left(1 - \frac{1}{p} \right)
		 + O_\epsilon \left( \left( \frac{\log Q}{\log x} \right)^{\epsilon} \right) \\
	&= \frac{(\log x) L_Q(1,f)}{2} \prod_{p\le Q} \left(1- \frac1p \right) + O_\epsilon(1) 
		\ll_\epsilon \frac{\log x}{\log Q} \cdot  |L_Q(1,f)|  + 1,
}
for all $x\ge Q^8$, since $L_Q(1,f)\ll1$ and $L'_Q(1,f)\ll_\epsilon\log Q$, by Lemma\ \ref{Lbound-2}(a), and $\gamma_{1,Q}\ll\log Q$, by Lemma \ref{zeta}. Moreover, $|L_Q(1,f)| \asymp_\epsilon \log Q/ \log Q'$ by\ \eqref{complex-dist-e11}. Thus we deduce that
\[
\sum_{\substack{ a\le\sqrt{x} \\ P^-(n)>Q }} \frac{(1*f)(a)}{a} \ll_\epsilon 1  \quad(Q^8\le x\le Q').
\]
In addition, Lemma \ref{complex-dist-2}, the Cauchy Schwarz inequality, and the inequality $|1+u|^2\le2\Re(1+u)$, for $u\in\{z\in\SC:|z|\le1\}$, imply that
\als{
\sum_{\substack{ a\le\sqrt{x} \\ P^-(n)>Q }} \frac{|(1*f)(a)|}{a}
	&\ll \exp\left\{ \sum_{Q<p\le \sqrt{x}}\frac{|1+f(p)|}{p}\right\} \\
	&\le \exp\left\{ \left(\sum_{Q<p\le x}\frac{1}{p}\right)^{1/2}
		\left(\sum_{Q<p\le \sqrt{x}}\frac{2(1+\Re(f(p)))}{p}\right)^{1/2} \right\} \\
	&\le \exp\left \{ O_\epsilon \left( \left( \log\frac{\log x}{\log Q}  \right)^{1/2}    \right) \right\} 
		\ll_\epsilon \left( \frac{\log x}{\log Q}  \right)^{\epsilon/2} ,	
}
for all $x\in[Q^8, Q']$. Therefore
\eq{complex-dist-e1000}{
&\sum_{\substack{ a\le\sqrt{x} \\ P^-(n)>Q }} \frac{(1*f)(a)}{a}
		\left\{ \overline{L_Q}(1,f)\prod_{p\le Q}\left(1- \frac1p \right)
		+ O\left(\frac{(\log Q)^{\epsilon}}{(\log x)^{1+\epsilon}} \right)  \right\} \\
	&\quad \ll_\epsilon \frac{ |L_Q(1,f)|}{\log Q} +  \frac{ (\log Q)^{\epsilon/2}}{(\log x)^{1+\epsilon/2}} 
	 	\ll_\epsilon \frac{1}{\log Q'} +  \frac{ (\log Q)^{\epsilon/2}}{(\log x)^{1+\epsilon/2}} .
}
for all $x\in[Q^8,Q']$, where we used\ \eqref{complex-dist-e11}. Similarly,
\eq{complex-dist-e1001}{
\left|  L_Q(1,f)\prod_{p\le Q}\left(1- \frac1p \right)
		+ O\left(\frac{(\log Q)^{\epsilon}}{(\log x)^{1+\epsilon}} \right)  \right|^2
	&\ll \frac{|L_Q(1,f)|^2}{(\log Q)^2} +  \frac{(\log Q)^{2\epsilon}}{(\log x)^{2+2\epsilon}} \\
	&\ll_\epsilon \frac{1}{(\log Q')^2} + \frac{(\log Q)^{2\epsilon}}{(\log x)^{2+2\epsilon}},  
}
by relation\ \eqref{complex-dist-e11}. Inserting \eqref{complex-dist-e1000} and \eqref{complex-dist-e1001} into\ \eqref{complex-dist-e12}, we deduce that
\[
\sum_{p\le x}(1+\Re(f(p))) \ll_\epsilon \frac{x}{\log Q'} +  \frac{x (\log Q)^{\epsilon/2}}{(\log x)^{1+\epsilon/2}} \quad(Q^8\le x\le Q').
\]
So the Cauchy Schwarz inequality and the inequality $|1+f(p)|^2\le2\Re(1+f(p))$ imply that
\[
\sum_{p\le x}|1+f(p)| \ll_\epsilon \frac{x}{(\log x)^{1/2}(\log Q')^{1/2}} +  \frac{x (\log Q)^{\epsilon/4}}{(\log x)^{1+\epsilon/4}} \quad(Q^8\le x\le Q').
\]
Finally, summation by parts yields the estimate
\[
\sum_{Q<p\le Q'}\frac{|1+f(p)|}{p} =O(1) + \sum_{Q^8 <p\le Q'}\frac{|1+f(p)|}{p}  \ll_\epsilon1,
\]
which completes the proof of Theorem \ref{complex-dist}.
\end{proof}


We conclude this section by showing Theorem \ref{min-dist}.

\begin{proof}[Proof of Theorem \ref{min-dist}] It suffices to show the theorem when $Q$ is large enough. Note that $\eta\ll1$, by Lemma \ref{Lbound-2}(a). Set $X=e^{1/(\sigma-1)}$. If $X\le Q$, then $|L_Q(\sigma+it,f)|\asymp1$ for all $t\in[-\tau,\tau]$, by Lemma \ref{dist-l1}, and the lemma follows. So for the rest of the proof we assume that $X>Q$. For each $t\in[-\tau,\tau]$, Theorem \ref{complex-dist} implies that there is some $C_t'\in[Q,+\infty]$ such that 
\eq{min-dist-e1}{
\sum_{Q<p\le C_t'} \frac{ |1+ f(p)p^{-it}| }{p} \ll_\epsilon1
	\quad\text{and}\quad
\sum_{C_t'<p\le z} \frac{f(p)}{p^{1+it}} \ll_{\epsilon} 1  \quad(z>C_t'  ).
}
In particular, $|L_Q(\sigma+it,f)|\asymp_\epsilon (\log Q)/\min\{\log X,\log C_t'\}$ by Lemma \ref{dist-l1}. So if we set $C_t=\min\{C_t',X\}\ge Q$, then $\eta\asymp_\epsilon(\log Q)/\log C_{t_0}$ and the theorem becomes equivalent to showing that
\eq{min-dist-e2}{ 
|L_Q(\sigma+it,f)| \asymp_\epsilon 
	\begin{cases} 
		(\log Q)/\log C_{t_0} & \text{if}\ |t-t_0|\le1/\log C_{t_0},\cr
		|t-t_0|\log Q & \text{if}\ 1/\log C_{t_0}\le |t-t_0|\le 1/\log Q,\cr  
		1  & \text{if}\ |t-t_0|\ge1/\log Q,
	\end{cases}
}
for all $t\in J$. So we shall prove this relation instead.

First, note that $C_t\le C_{t_0}^{O_\epsilon(1)}$ for every $t\in J$, by the choice of $t_0$. Thus if $|t-t_0|\le1/\log C_{t_0}$, then relation \eqref{min-dist-e1} and the formula $p^{i(t-t_0)}=1+O(|t-t_0|\log p)$, which is valid for $p\le C_{t_0}\le e^{1/|t-t_0|}$, imply that
\als{
\sum_{Q<p\le X}  \frac{\Re(f(p)p^{-it})}p 
	& =  O_\epsilon(1) + \sum_{Q<p\le C_{t_0} }\frac{\Re(f(p)p^{-it})}p \\
	& =  O_\epsilon(1) + \sum_{Q<p\le C_{t_0} }\frac{\Re(f(p)p^{-it_0})}{p}  
		+ O\left( \sum_{Q<p\le C_{t_0}} \frac{ |t-t_0|\log p }{p} \right)	\\
	&=  O_\epsilon(1) -   \sum_{Q<p\le C_{t_0}}\frac1p
		= O_\epsilon(1) -  \log\left( \frac{\log C_{t_0}}{\log Q} \right).
}
So Lemma \ref{dist-l1} completes the proof of \eqref{min-dist-e2} in this case.

Fix now $t\in J$ with $|t-t_0|\ge1/\log C_{t_0}$ and let $y= \max\{Q,e^{1/|t-t_0|}\} $. Then
\als{
\sum_{y<p \le C_{t_0} } \frac{\Re(f(p)p^{-it}) }{p}
	&= - \sum_{y<p\le C_{t_0} } \frac{\Re(p^{ i(t_0-t)} ) }{p} 
	 	+ O\left( \sum_{y<p\le C_{t_0} }\frac{|p^{it_0}+f(p)|}{p} \right) \\
	& = - \log\left| \zeta_y\left( 1+ \frac{1}{\log C_{t_0}}  +i(t-t_0) \right) \right| + O_\epsilon(1),
}
by \eqref{min-dist-e1} with $t=t_0$ and Lemma \ref{dist-l1}. Moreover, letting $x\to\infty$ in the second part of Lemma \ref{zeta}, which is applicable because $y\ge Q\ge (V_{2\tau}) ^{100} \ge (V_{t-t_0})^{100}$ by assumption, implies that
\als{
\zeta_y\left( 1+ \frac{1}{\log C_{t_0}} +i(t-t_0) \right)
	&= \left(\frac{1}{1/\log C_{t_0}+ i(t-t_0)} + \gamma_{1+1/\log C_{t_0}+i(t-t_0),y}\right) 
		\prod_{p\le y}\left(1-\frac1p \right)   \\
	&\ll 1,
}
since $|t-t_0|\ge1/\log y$ and $\gamma_{1+1/\log C_{t_0}+i(t-t_0),y}\ll \log y$. Consequently, we deduce that
\[
\sum_{y<p\le C_{t_0}}\frac{ \Re(f(p)p^{-it})}{p} \ge - c_\epsilon,
\]
for some constant $c_\epsilon$ that depends at most on $\epsilon$, which, together with Lemma \ref{dist-l2}, yields the estimate
\[
\sum_{y<p\le C_{t_0}}\frac{ \Re(f(p)p^{-it})}{p} \ll_\epsilon 1.
\]
Since we also have that $C_{t}\le C_{t_0}^{O_\epsilon(1)}$, we conclude that
\eq{min-dist-e3}{
\sum_{y<p\le X}\frac{\Re(f(p)p^{-it} )}{p } 
	= O_\epsilon(1) + \sum_{C_{t_0}<p\le X}\frac{\Re(f(p) p^{-it} )}{p } 
	\ll_\epsilon 1  ,
}
by \eqref{min-dist-e1}. If $|t-t_0|\ge1/\log Q$, so that $y=Q$, the above relation and Lemma \ref{dist-l1} yield \eqref{min-dist-e2}. Finally, if $1/\log C_{t_0}\le |t-t_0|\le 1/\log Q$, then \eqref{min-dist-e3} implies that
\als{
\sum_{Q<p\le X}\frac{\Re(f(p)p^{-it} )}{p}
	&= O_{\epsilon}(1) +  \sum_{Q<p\le e^{1/|t-t_0|}}  \frac{\Re(f(p)p^{-it})}{p} \\
	&= O_{\epsilon}(1) +  \sum_{Q<p\le e^{1/|t-t_0|}}  \frac{\Re(f(p)p^{-it_0})}{p}
		+ O\left(\sum_{p\le e^{1/|t-t_0|} } \frac{|t-t_0|\log p}{p} \right) \\
	&= O_{\epsilon}(1) +  \sum_{Q<p\le e^{1/|t-t_0|}}  \frac{\Re(f(p)p^{-it_0})}{p}.
}
Since $e^{1/|t-t_0|}\le C_{t_0}\le C_{t_0}'$, applying relation \eqref{min-dist-e1} with $t=t_0$, 
we obtain the estimate
\[
\sum_{Q<p\le X}\frac{\Re(f(p)p^{-it} )}{p} 
	= O_{\epsilon}(1) - \sum_{Q<p\le e^{1/|t-t_0|}}  \frac{1}{p} 
	= O_{\epsilon}(1) + \log(|t-t_0|\log Q)  .
\]
Together with Lemma \ref{dist-l1}, this implies relation \eqref{min-dist-e2} in this last case too, thus completing the proof of the theorem.
\end{proof}


\section{Proof of Theorems\ \ref{complex} and \ref{power}: the case when $L(1+it_0,f)=0$}\label{L(1+it_0)=0}

In this section we show part (b) of Theorems\ \ref{complex} and \ref{power}. These proofs are distinctly different from the proofs of parts (a).

\begin{proof}[Proof of Theorem \ref{complex}(b)] For $x\ge Q_t = Q^{2(1+|t|)^{\frac{1}{A-2}} } \ge Q^2$, the argument leading to \eqref{Lbound-1-it} and our assumption that $f$ satisfies \eqref{small-Q} imply that
\eq{small-Qt}{
\sum_{n\le x} f(n) n^{-it} 
	\ll \sqrt{x} + \frac{x(\log Q_t)^{A-2}}{(\log x)^{A}}
	\ll x\cdot  \frac{A^A(\log Q_t)^{A-2}}{(\log x)^A}, 
}
since $\sqrt{x}\ge(e\log x)^A/(2A)^A$. We set $h(n)=f(n)n^{-it_0}$ and follow the argument in the proof of Theorem\ \ref{complex-dist} with $h$ in place of $f$ and $Q_{t_0}$ in place of $Q$ to deduce that
\als{
\sum_{Q<p\le x}( 1+\Re(h(p)) ) 
	&\le \Re\left( \sum_{\substack{ a\le\sqrt{x} \\ P^-(n)>Q_{t_0} }} (1*h)(a) 
		\sum_{\substack{ b\le x/a \\ P^-(b)>Q_{t_0} }}(1*\overline{h} )(b) \right)
		- \frac{1}{2} \left| \sum_{\substack{ a\le\sqrt{x} \\ P^-(a)>Q_{t_0} }} (1*h)(a) \right|^2.
}
Now, note that $L(1,h)=L(1+it_0,f)=0$ and therefore $L_{Q_{t_0} } (1,h)=0$. Together with Lemma\ \ref{Lbound-1}(b), this implies that
\[
\sum_{\substack{ n\le u \\ P^-(n)>Q_{t_0} }}(1*h)(n) 
	= u L_{ Q_{t_0} }(1,h)\prod_{p\le Q_{t_0} }\left(1- \frac1p \right)
		+ O\left(\frac{u (8\log Q_{t_0} )^{A-2}  } {(\log u)^{A-1}} \right) 
	\ll \frac{ u (8\log Q_{t_0} )^{A-2} }{(\log u)^{A-1}},
\]
for all $u\ge Q_{t_0}^4$. Hence
\eq{complex-pf-e0}{
\sum_{Q_{t_0}<p\le x}(1+\Re(h(p))) 
	\ll \frac{x(16\log Q_{t_0})^{A-2}}{(\log x)^{A-1}} 
		\sum_{\substack{ a\le\sqrt{x} \\ P^-(n)>Q_{t_0} }} \frac{|(1*h)(a)|}{a} 
		+\frac{x(16\log Q_{t_0})^{2A-4}}{(\log x)^{2A-2}},
}
for all $x\ge Q_{t_0}^8 $. Lastly, relation\ \eqref{small-Qt} allows us to apply Theorem\ \ref{complex-dist} with $h$ in place of $f$ and $Q_{t_0}^{O_\epsilon(1)}$ in place of $Q$. Since, $L(1,h)=0$, the parameter $Q'$ in Theorem\ \ref{complex-dist} is equal to $\infty$. Consequently,
\[
\sum_{\substack{ a\le\sqrt{x} \\ P^-(n)>Q_{t_0} }} \frac{|(1*h)(a)|}{a} 
	\ll \exp\left\{ \sum_{Q_{t_0} <p\le \sqrt{x} } \frac{|1+h(p)|}{p} \right\} \ll_{\epsilon} 1 \quad(x\ge Q_{t_0}^2).
\]
Inserting this estimate into\ \eqref{complex-pf-e0}, we find that
\[
0\le \sum_{p\le x}(1+\Re(h(p))) 
	\le 2Q_{t_0}  + \sum_{Q_{t_0}<p\le x}(1+\Re(h(p))) 
	\ll_{A,\epsilon} \frac{x(\log Q_{t_0})^{A-2}}{(\log x)^{A-1}} 
		 \quad (x\ge Q_{t_0}^{16}).
\]
This inequality holds trivially when $x\in[2,Q_{t_0}^{16}]$ too. So the theorem follows by partial summation.
\end{proof}


\begin{proof}[Proof of Theorem \ref{power}(b)] Since $f$ assumes values in $[-1,1]$, we have that
\eq{power-b-e1}{
0\le \sum_{Q<p\le x}( 1+ f(p) ) \le  \sum_{\substack{ n\le x  \\  P^-(n)>Q }} (1*f)(n) \quad( x\ge Q).
}
As $L(1,f)=0$, applying Lemma \ref{Lbound-1}(b) with $\sigma_0=1-\delta\in [3/5,1-1/\log Q]$, $A=2$ and $M=1$, which is possible by \eqref{small-power}, yields the estimate
\[
\sum_{\substack{ n\le x  \\  P^-(n)>Q }} (1*f)(n) 
	\ll \frac{ x^{\frac{3+\delta}{4} } } {\log x}
		+ \frac{ x^{1- \frac{1}{60\log Q}}}{\log Q}   
	\ll  \frac{ x^{1- \frac{1}{60\log Q}}}{\log Q}   \quad(x \ge Q).
\]
Inserting the above estimate into \eqref{power-b-e1} and using partial summation, we conclude that
\als{
0\le \sum_{p\le x} (1+ f(p) ) \log p
	= O(Q) + \sum_{Q<p\le x} (1+ f(p)) \log p 
	\ll Q + \frac{ x^{1- \frac{1}{60\log Q} } \log x}{\log Q}   
		\ll x^{1- \frac{1}{61\log Q} } ,
}
for all $x\ge Q$, which completes the proof.
\end{proof}


\section{Real zeroes and the size of $L(1,f)$}\label{l(1,f)}

In this section we prove Theorem \ref{siegel}.

\begin{proof}[Proof of Theorem \ref{siegel}] For $\sigma>1-1/\log Q$ and $y\ge1$, we have that
\eq{siegel-e1}{
L_y(\sigma,f) 
	=  \lim_{N\to\infty}  \sum_{n\le N}\frac{f(n)}{n^\sigma} \sum_{\substack{d|n\\P^+(d)\le y}}\mu(d)
	&= \sum_{P^+(d)\le y} \frac{\mu(d)f(d)}{d^\sigma}  \lim_{N\to\infty} \sum_{n\le N}\frac{f(n)}{n^s}  \\
	&=  L(\sigma,f)  \prod_{p\le y} \left(1-\frac{f(p)}{p^\sigma}  \right),
} 
by our assumption that $f$ is totally multiplicative. In particular, the zeroes of $L(s,f)$ and $L_Q(s,f)$ in the region $\Re(s)>1-1/\log Q$ are in one-to-one correspondence. Moreover, for $y\ge Q$ we have that
\eq{siegel-e2}{
L_y(\sigma,f) =  L_Q(\sigma,f)   \prod_{Q<p\le y}\left( 1 - \frac{f(p)}{p^\sigma} \right).
}

Next, by Lemma \ref{zeta}, there is a constant $M\ge120$ such that $\gamma_{1-\eta,y}\in[-M\log y,M\log y]$ for all $y\ge 3$ and all $\eta\in[0,1/(60\log y)]$. We claim that, for $0\le\eta\le1/(M\log Q)$, we have the relation 
\eq{siegel-e3}{
L_Q(1-\eta,f) \ge 0  
	\quad\implies\quad 
L_Q(1,f) \gg  \eta \log Q.
} 
Indeed, our assumption that \eqref{small-power} holds with $\delta=1-1/\log Q$ implies that
\eq{small-power-2}{
\sum_{n\le x} f(n) \ll \frac{(\log Q)x^{1-\frac{1}{2\log Q}}}{(\log x)^3}  \quad(x \ge Q).
}
Thus Lemma \ref{Lbound-1}(b) with $A=3$, $\sigma_0=1-1/(2\log Q)$, $M\asymp\log Q$, $\sigma=1-\eta$, $y=e^{1/(M\eta)}$ and $x=y^{C}$, where $C$ is a large enough constant, yields the formula
\eq{siegel-e1}{
\sum_{\substack{n\le e^{C/\eta} \\  P^-(n)>y } }  \frac{(1*f)(n)}{n^{1-\eta}}
	& = \left\{\left( - \frac{1}{\eta} +  \gamma_{1-\eta,y}\right)  L_y(1-\eta,f)
		+   \frac{e^CL_y(1,f)}{\eta}  \right\}   \prod_{p\le y}\left( 1 - \frac{1}{p} \right)  
		+   O\left( \frac{1}{C} \right) .
}
By the choice of $y$, we have that $\gamma_{1-\eta,y}\le M\log y=1/\eta$. Moreover, since $f$ is real valued, if $L_Q(1-\eta,f)\ge0$, then $L_y(1-\eta,f)\ge0$ by relation \eqref{siegel-e2}. So the term $(-1/\eta+\gamma_{1-\eta,y})L_y(1-\eta,f)$ is non-positive. Consequently, choosing $C$ large enough in \eqref{siegel-e1} gives us that
\[
\sum_{\substack{n\le e^{C/\eta} \\  P^-(n)>y } }  \frac{(1*f)(n)}{n^{1-\eta}}
	\le \frac{e^CL_y(1,f)}{\eta}    \prod_{p\le y}\left( 1 - \frac{1}{p} \right) + \frac{1}{2}.
\]
On the other hand, the sum on the left hand side of the above inequality is at least $(1*f)(1)=1$, by positivity (our assumption that $f$ is a real valued completely multiplicative function implies that $(1*f)(n)\ge0$ for all $n$). So we find that
\eq{siegel-e2}{
L_y(1,f) 	\ge  \frac{\eta}{2e^C} \prod_{p\le y} \left( 1 - \frac{1}{p} \right)^{-1} 
		\asymp \eta  \log y  \asymp1.
}
However, if $Q'$ is as in Theorem \ref{complex-dist}, then $L_y(1,f)\asymp (\log y)/\log(yQ')$. Comparing this estimate with \eqref{siegel-e2}, we find that $\log y\gg\log Q'$. Since we also have that $\log Q'\asymp(\log Q)/L_Q(1,f)$ and $\log y\asymp1/\eta$, \eqref{siegel-e3} follows.

Fix now a small enough constant $c\le 1/M^2$. Note that $L_Q(\sigma,f)\ge0$ for $\sigma>1$, by the Euler product representation. So if $L_Q(s,f)$ does not vanish in $[1-\sqrt{c}/\log Q,1]$, then we must have that $L(1-\sqrt{c}/\log Q,f)>0$ by continuity, and \eqref{siegel-e3} gives us that $L_Q(1,f)\ge c_1\sqrt{c}$ for some positive constant $c_1$ that is independent of $c$. Consequently, Lemma \ref{Lbound-2}(b) implies that, for $\sigma\in [1-c/\log Q,1+c/\log Q]$, 
\eq{siegel-e4}{
L_Q(\sigma,f) 	=   L_Q(1,f)  +   \int_1^\sigma L'_Q(u,f)du
			& =   L_Q(1,f)  +  O(|1-\sigma|\log Q)\\
			& \ge c_1\sqrt{c}+O(c)\gg\sqrt{c},
}
provided that $c$ is small enough. Since we also have that $L_Q(\sigma,f)\ll1$ by Lemma \ref{Lbound-2}(b), the theorem follows in this case.

Lastly, consider the case that $L_Q(s,f)$ has a zero in $[1-\sqrt{c}/\log Q,1]$, say at $\beta$. Relation \eqref{small-power-2} allows us to apply Lemma \ref{Lbound-1}(b) with $A=3$, $\sigma_0=1-1/(2\log Q)$, $M\asymp\log Q$, $x=Q^{1/c^{1/4}}$, $y=Q$ and $s=1$ to obtain the estimate 
\als{
\sum_{\substack{n\le Q^{1/c^{1/4}} \\  P^-(n)>Q  }}\frac{(1*f)(n)}n
	&=\left\{  \left(\frac{\log Q}{c^{1/4}}  +   \gamma_{1,Q}\right)L_Q(1,f)
		+   L_Q'(1,f)\right\}  \prod_{p\le Q}\left(1-\frac1p\right)	  +  O\left(c^{1/4}\right)  
}
Relation \eqref{siegel-e4} with $\sigma=\beta$ and Lemma \ref{Lbound-2}(b) imply that
\[
L_Q(1,f)=\int_\beta^1L_Q'(u,f)du \ll  (1-\beta)\log Q\ll\sqrt{c}.
\]
Moreover, $\gamma_{1,Q}\ll\log Q$ by Lemma \ref{zeta}. Combining the above estimates, we deduce that
\[
\sum_{\substack{n\le Q^{1/c^{1/4}} \\  P^-(n)>Q  }}\frac{(1*f)(n)}n
	=  L_Q'(1,f)  \prod_{p\le Q}  \left(1-\frac{1}{p}  \right)  +   O\left(c^{1/4}\right),
\]	
However, as before, the left hand side of the above inequality is $\ge(1*f)(1)=1$, by positivity. So if $c$ is small enough, then $L_Q'(1,f)\ge c_0\log Q$ for some absolute positive constant $c_0$. Consequently, for any $u\in[1-c/\log Q,1+c/\log Q]$, Lemma \ref{Lbound-2}(b) implies that 
\[
L_Q'(u,f) 	= L_Q'(1,f) -  \int_u^1  L''_Q(w,f)d w
		\ge c_0\log Q  +    O\left(  |1 - u |  \log^2Q\right)
		\ge \frac{c_0\log Q}2,
\] 
provided that $c$ is small enough. Since we also have that $L_Q'(u,f)\ll \log Q$, by Lemma \ref{Lbound-2}(b), we deduce that 
\[
\frac{c_0}{2} \log Q\le L_Q'(u,f)\le c_0' \log Q,
\]
where $c_0'$ is some positive constant. The theorem then follows by the identity $L_Q(\sigma,f)=\int_\beta^\sigma L_Q'(u,f)du$, which holds for all $\sigma>1-1/\log Q$.
\end{proof}


\section{Bounds for the derivatives of $\frac{L'}{L}(s,f)$ and $\frac{1}{L}(s,f)$}\label{1/l(s,f)}

In this section we list some estimates for the derivatives of $\frac1{L}(s,f)$ and $\frac{L'}{L}(s,f)$, which we shall need in the proof of Theorems \ref{complex}(a) and \ref{power}(a). The key lemma is the following result which has a combinatorial flavour and is based on an idea in \cite[p. 40]{IK}, also exploited in \cite[Lemma 2.1]{K}.

\begin{lemma}\label{der-lemma} Let $k\in\SN$, $D$ be an open subset of $\SC$, $s\in D$, and $F:D\to\SC$ be a function which is differentiable $k$ times at $s$. Assume that $F(s)\neq0$ and set 
\[
M = \max_{1\le j\le k}\left\{\frac1{j!}
	\left|  \frac{F^{(j)}}F(s)  \right|\right\}^{1/j} 
		\quad\text{and}\quad
	N =  \max_{1\le j\le k} \left\{    \frac1{j!}  \left|\left(  \frac{F'}F \right)^{(j-1)}(s)  \right|\right\}^{1/j}.
\] 
Then $M/2\le N\le2M$.
\end{lemma}


\begin{proof} By arguing as in \cite[Lemma 2.1]{K}, we find that $N\le2M$.

In order to show that $M\le2N$, we argue inductively. First, we have that $|(F'/F)(s)|\le N$, by the definition of $N$. Next, we assume that $|F^{(j)}(s)/F(s)|\le j!(2N)^j$ for all $j\in\{1,\dots,r\}$, for some $r\in\{1,\dots,k-1\}$. Since 
\[
F^{(r+1)}(s) 
	=  \left( F \cdot \frac{ F' } F \right)^{(r)}(s)
	=\sum_{j=0}^r \binom rj F^{(j)}(s)  \left(\frac{F'}F\right)^{(r-j)}(s),
\]
we find that 
\[
\left| \frac{ F^{(r+1)} }{F}(s) \right|
	\le (r+1)! \sum_{j=0}^r  (2N)^j N^{r-j+1}
	< (r+1)!(2N)^{r+1},
\]
which completes the inductive step and hence the proof of the lemma.
\end{proof}


Using the above lemma, we bound the derivatives of $\frac{L'}{L}(s,f)$ and of $\frac{1}{L}(s,f)$ in terms of the derivatives of $L(s,f)$ and a lower bound on $|L(s,f)|$ in a similar fashion as in \cite[p. 40-42]{IK}. Similar arguments were also used in \cite{K}.

\begin{lemma}\label{der-bound} Let $s=\sigma+it$ with $\sigma>1$ and $t\in\SR$, $k\in\SN$, $Q\ge2$ and $M\ge1$. Consider a completely multiplicative function $f:\SN\to\SU$ such that 
\[
|L_Q^{(j)}(s,f)| \le j!M^j   \quad(1\le j\le k).
\] 
There is an absolute constant $c$ such that for $z\ge3/2$ we have that 
\[
\left| \left(  \frac{L_z'}{L_z} \right)^{(k-1)}(s,f) \right|  
+ \left| \frac{(1/L_z)^{(k)}}{1/L_z}(s,f) \right|  
	\ll c^k k!\left(\frac{M} { \min\{1,|L_Q(s,f)| \}  }  +    \log(zQ)   \right)^k.
\]
\end{lemma}


\begin{proof} Note that 
\eq{der-e1}{
\left(\frac{L_z'}{L_z}\right)^{(j-1)}(s,f) 
	= \left(\frac{L_Q'}{L_Q}\right)^{(j-1)}(s,f)+O( j! (c_1 \log(zQ))^j )
		\quad(1\le j\le k)
}
and 
\[
\frac1{j!}\left|\frac{L_Q^{(j)}}{L_Q}(s,f)\right|\le\frac{M^j}{|L_Q(\sigma,f)|}
	\le  \left(\frac M{\min\{|L_Q(s,f)|,1\}}\right)^j\quad(1\le j\le k).
\]
So Lemma \ref{der-lemma}, applied with $F(s)=L_Q(s,f)$, implies that 
\[
\left|\left(  \frac{L_Q'}{L_Q}\right)^{(j-1)}(s)  \right| 
	\le j!\left(\frac{2M}{\min\{1,|L_Q(s,f)|\}}\right)^j \quad(1\le j\le k).
\]
Together with \eqref{der-e1}, this implies that
\eq{der-e2}{
\left|\left(  \frac{L_z'}{L_z}\right)^{(j-1)}(s)  \right| 
	\le c_2^j j!\left(\frac{M} { \min\{1,|L_Q(s,f)| \}  }  +    \log(zQ)   \right)^j
	\quad(1\le j\le k),
}
for some constant $c_2\ge2$. Finally, since the logarithmic derivative of a function $f$ is minus the logarithmic derivative of its inverse $1/f$, relation\ \eqref{der-e2} and Lemma \ref{der-lemma} applied with $F(s)=1/L_z(s,f)$ imply that
\[
 \left| \frac{(1/L_z)^{(k)}}{1/L_z}(s,f) \right|  
	\le 2^kc_2^k k!\left(\frac{M} { \min\{1,|L_Q(s,f)| \}  }  +    \log(zQ)   \right)^k ,
\]
and the lemma follows by taking $c=2c_2$.
\end{proof}


\section{Proof of Theorems\ \ref{complex} and \ref{power}: the case when $L(1+it,f)$ does not vanish}\label{complex-proof}

In this section, we complete the proof of Theorems\ \ref{complex} and \ref{power} by showing their first parts. This will be accomplished in two steps. We start with two preliminary estimates in Subsection \ref{complex-prelim}, which are then combined to show our main theorems in Subsection \ref{complex-completion}.

Here and for the rest of this paper, given an arithmetic function $g: \SN\to\SC$, $k\ge0$, $x\ge1$ and $\sigma>1$, we set 
\[
S_k(x;g) =  \sum_{n\le x}g(n)(\log n)^k
	\quad\text{and}\quad 
I_k(\sigma;g) =  \left(  \int_0^\infty  \left|  \frac{S_k(e^u;g)}{e^{\sigma u}}  \right|^2  dt  \right)^{1/2}.
\]


\subsection{Preliminary estimates}\label{complex-prelim}

\begin{prop}\label{complex-prop-3} Let $\epsilon>0$ and $Q\ge3$. Consider a completely multiplicative function $f:\SN\to\SU$ satisfying \eqref{small-Q} for some $A\ge2+\epsilon$. For $\sigma=1+1/\log x$, we have that
\als{
\frac{I_k(\sigma;\Lambda f)}{ (\log x)^{k+ \frac{1}{2}} }
	&\ll_{A,\epsilon} (N(x;T)-\log\log Q)
		\left( \frac{\log Q}{e^{N(x;T)} } \right)^{ \min\left\{A-2, k + \frac{1}{2}\right\} }  \\	
	&\quad +  \left(\log \frac{\log x}{\log Q} \right)
	\left( \frac{\log Q}{\log x} \right)^{ \min \left\{ k+\frac{1}{2} - \frac{(k+1)(k+2)}{4(A-1)} , 
	\frac{(2k+1)(A-2)}{2k+A-1}  \right\} }
	+   \frac{1}{T} ,
}
uniformly in $k\in[0,A-2]\cap \SZ$, $x\ge Q^4$ and $T\ge1$, where $N(x;T)$ is defined by \eqref{N(x;T)}.
\end{prop}


\begin{proof} Without loss of generality, we may assume that $x$ is large enough in terms of $A$ and $\epsilon$. Note that $h(\tau) = \tau/e^{(A-2)N(x;\tau)}$ is an increasing function of $\tau$ that tends to infinity as $\tau\to\infty$ and is continuous from the right. Moreover, $h(1)\le 1/e^{(A-2)\log(2\log Q)} =  1/(2\log Q)^{A-2}$. So there exists a unique $T_1>1$ such that $h(\tau)<1/(\log Q)^{A-2}$ for all $\tau<T_1$ and $h(T_1)\ge 1/(\log Q)^{A-2}$. We may assume that $T<T_1$. Indeed, if $T\ge T_1$, then set $T_1'=\max\{T_1-1,1\}$ and note that
\[
(N(x; T_1')-\log\log Q)
		\left( \frac{\log Q}{e^{N(x; T_1' )} } \right)^{ \min\left\{A-2, k + \frac{1}{2}\right\} }  
	\ll_\epsilon (N(x;T)-\log\log Q)
		\left( \frac{\log Q}{e^{N(x;T)} } \right)^{ \min\left\{A-2, k + \frac{1}{2}\right\} }  ,
\]
since $N(x;T)\le N(x; T_1')$, and that
\[
\frac{1}{T_1'} 
	\le \frac{2}{T_1} 
	\le 2\cdot \left( \frac{\log Q}{ e^{N(x;T_1)} }  \right)^{A-2}
	\le 2\cdot \left( \frac{\log Q}{ e^{N(x;T)} }  \right)^{A-2},
\]
since $h(T_1)\ge 1/(\log Q)^{A-2}$ and $N(x;T)\le N(x;T_1')$. So the claimed result follows by the case when $T=T_1'\in[1,T_1)$. Similarly, we may assume that 
\[
T\le T_2 := \left( \frac{\log x}{2\log Q} \right)^{ \frac{(2k+1)(A-2)}{2k+A-1} } -1
	\le \left( \frac{\log x}{2\log Q} \right)^{ A-2}  -1 
\]
In particular, $Q_t\le x$ for all $t\in[-T,T]$.

\medskip

Our starting point towards the proof of the proposition is the identity
\eq{plancherel}{
I_k(\sigma;\Lambda f)^2 = \int_0^\infty \left|  \frac{S_k(e^u;\Lambda f)}{e^{\sigma u}} \right|^2 du
	=  \frac1{2\pi} \int_{\SR}\left|\left(\frac{L'}{L}\right)^{(k)}(\sigma+it,f)\right|^2  \frac{dt}{\sigma^2+t^2},
} 
which follows by observing that the Fourier transform of the function $u\to e^{-\sigma u}S_k(e^u;\Lambda f)$ is the function $\xi \to (-1)^k(-L'/L)^{(k)}(\sigma+2\pi i\xi,f)/(\sigma+2\pi i\xi)$ and then applying Plancherel's identity. This turns the proposition to a problem of bounding $(L'/L)^{(k)}(s,f)$ on average. However, unlike the proof of Hal\'asz's theorem (i.e. Theorem \ref{halasz}), where a similar integral is bounded, with $L'(s,f)$ in place of $(L'/L)^{(k)}(s,f)$, we do not have at our disposal the factorization $L'(s,f)= (L'/L)(s,f) \cdot L(s,f) $ that allows one to control the integral of $|L'(s,f)|^2$ without losing any logarithmic factors. So we are forced to use a different strategy, as we will see below.

In order to bound the integral on the right hand side of \eqref{plancherel}, we let $J_\tau=\{t\in\SR:\tau/2-1\le|t|\le\tau-1\}$ and estimate the contribution to this integral from $t\in J_\tau$, for each $\tau\ge2$. Observe that $\sigma^2 + t^2 \asymp \tau^2$ for $t\in J_\tau$, so it suffices to bound the integral $\int_{J_\tau} | (L'/L)^{(k)} (\sigma+it,f) |^2 dt $. First, note the `trivial' bound 
\eq{complex-3-e1}{
\int_{J_\tau} \left| \left( \frac{L'}{L} \right)^{(k)}(\sigma+it,f) \right|^2 dt
	&\le 3\int_{-\tau}^{\tau} \left| \left( \frac{\zeta'}{\zeta} \right)^{(k)}(\sigma+it) \right|^2 dt \\
	&\ll_A (\log x)^{2k+1} +  \tau(\log V_\tau)^{2k+2} ,
}
which follows by Lemmas \ref{mont} and \ref{lambda}. Moreover, if $\tau\le T$, then we claim that
\al{
\int_{J_\tau} \left|  \left(  \frac{L'}{L}  \right)^{(k)}( \sigma+it,f ) \right|^2 dt
	&\ll_{A,\epsilon} (N(x;T) - \log\log Q) \left(\frac{(\log x)(\log Q_\tau)} { e^{N(x;T)} }\right)^{2k+1} \nn
	&\quad +   \tau (\log Q_\tau)^{2k+1} \left(\log\frac{\log x}{\log Q}\right)  
		+ \tau (\log V_\tau)^{2k+2} \label{complex-3-e2} \\
	 &\quad +  (\log Q)^{2k+1} \sum_{j=1}^{k+1} \tau^{\frac{2(k-j+1)}{A-2}}  I_{j,\tau}, \nonumber
}
where
\[
I_{j,\tau} =  \min\left\{
			\left( \frac{\log x}{\log Q}\right)^{2j-1}  ,
			\frac{ \tau^{\frac{2}{A-2}} }{ (\log Q)^{2j-1}}\int_{J_\tau}|L_{Q}^{(j)}(\sigma+it,f)|^2dt   
		\right\}.
\]

Before we proceed to the proof of \eqref{complex-3-e2}, we show how to combine relations \eqref{complex-3-e1} and \eqref{complex-3-e2} in order to complete the proof of the Proposition \ref{complex-prop-3}. We partition the range of integration on the right hand side of \eqref{plancherel} as $\bigcup_{m\ge0}J_{2^m}$ and apply \eqref{complex-3-e1} or \eqref{complex-3-e2} to the part of the integral over $J_{2^m}$ according to whether $2^m>T$ or $2^m\le T $. So
\als{
 I_k(\sigma;\Lambda f)^2
	&\ll_{A,\epsilon} \sum_{2^m\le T} \left\{ \frac{(N(x;T) - \log\log Q)}{4^m} 
		\left(\frac{(\log x)(\log Q_{2^m})} { e^{N(x;T)} }\right)^{2k+1} \right. \\
	&\quad\qquad\qquad \left. +   \frac{(\log Q_{2^m})^{2k+1} \log\frac{\log x}{\log Q} }{2^m} 
		+ \frac{(\log V_{2^m})^{2k+2}}{2^m}\right\} \\
	&\qquad	+  \sum_{j=1}^{k+1} \sum_{ 2^m\le T  } 
			\frac{(\log Q)^{2k+1} I_{j,2^m} }{ 4^{m\left(1-\frac{k-j+1}{A-2}\right)}}  
		+ \sum_{2^m>T }  \left\{ \frac{(\log x)^{2k+1}}{4^m} + \frac{ (\log V_{2^m})^{2k+2}}{2^m} \right\}, 
}
Since $\log Q_{2^m}\asymp 2^{\frac{m}{A-2}}\log Q$, we find that
\als{
I_k(\sigma;\Lambda f)^2
	&\ll_{A,\epsilon} (\log T) \left( 1+ T^{\frac{2k+1}{A-2} -2 } \right) 
		(N(x;T) - \log\log Q) \left(\frac{(\log x)(\log Q)} { e^{N(x;T)} }\right)^{2k+1}  \\
	&\quad + (\log T) \left( 1+ T^{\frac{2k+1}{A-2} -1 } \right) (\log Q)^{2k+1}
		\left(\log \frac{\log x}{\log Q} \right)   + \frac{(\log x)^{2k+1} }{T^2}   \\
	&\quad +  (\log Q)^{2k+1}\sum_{j=1}^{k+1} \sum_{ 2^m\le T } 
			\frac{I_{j,2^m} }{ 4^{m\left(1-\frac{k-j+1}{A-2}\right)}}  .
}
We use our assumption that $T< T_1$, which implies that $T< (e^{N(x;T)}/\log Q)^{A-2}$, to bound $ (\log T) (1+ T^{\frac{2k+1}{A-2} -2 } )$ and our assumption that $T \le T_2 $ to bound $ (\log T) (1+ T^{\frac{2k+1}{A-2} -1 } )$. This yields that
\al{
\frac{I_k(\sigma;\Lambda f)^2}{(\log x)^{2k+1}}
	&\ll_{A,\epsilon}  (N(x;T) - \log\log Q)^2 \left(\frac{ \log Q} { e^{N(x;T)} }\right)^{ \min\{ 2k+1, 2(A-2) \} }  \nn
	&\quad + 	\left(\log \frac{\log x}{\log Q} \right)^2 
		\left( \frac{\log Q }{ \log x } \right)^{  \min \left\{ 2k+1 , \frac{2(2k+1)(A-2)}{2k+A-1} \right\}  } 
		+ \frac{1}{T^2}  \label{complex-3-e4} \\
	&\quad	+  \left( \frac{\log Q}{\log x} \right)^{ 2k+1 }  
		\sum_{j=1}^{k+1} \sum_{ 2^m\le T } 
			\frac{I_{j,2^m} }{ 4^{m\left(1-\frac{k-j+1}{A-2}\right)}}  \nonumber. 
}
Finally, we bound the sum over $m$ appearing on the right hand side of \eqref{complex-3-e4} for each $j\in\{1,\dots,k+1\}$. So fix such a $j$ and note that, for any $L_j\ge4$, we have that
\[
\frac{I_{j,2^m}}{4^{m\left(1-\frac{k-j+1}{A-2}\right)}}
	\ll \begin{cases}
		\ds  \frac1{4^{m\left(1-\frac{k-j+1}{A-2}\right)} }   \left(\frac{\log x}{\log Q}\right)^{2j-1}   &\text{if}\ L_j<4^m\le T^2,\cr
				\cr
		\ds  L_j^{\frac{k-j+2}{A-2}}
			\int_{J_{2^m}} \frac{ |L_{Q}^{(j)}(\sigma+it,f)|^2}{(\log Q)^{2j-1} } \frac{dt}{\sigma^2+t^2}   
				&\text{if}\ 1<4^m\le L_j.
	     \end{cases}
\]
Since $1 - (k-j+1)/(A-2)\ge 1- k/(A-2)\ge0$, we deduce that
\eq{complex-3-e5}{
\sum_{ 2^m\le T  } \frac{I_{j,2^m}}{4^{m\left(1-\frac{k-j+1}{A-2}\right)}}
	&\ll \frac{\log(2T)}{L_j^{1 - \frac{k-j+1}{A-2} }  }  \left(\frac{\log x}{\log Q}\right)^{2j-1}
		+  L_j^{\frac{k-j+2}{A-2}}
			\int_{\SR}  \frac{ |L_{Q}^{(j)}(\sigma+it,f)|^2}{(\log Q)^{2j-1}} \frac{dt}{\sigma^2+t^2} 
}
In addition, Plancherel's identity yields that
\eq{complex-3-e6}{
\int_{\SR}  |L_{Q}^{(j)}(\sigma+it,f)|^2  \frac{dt}{\sigma^2+t^2} 
	=  2\pi  \int_{\log Q}^\infty   \left|\frac{S_j(e^u;f_{Q})}{e^{\sigma u}} \right|^2  du,
} 
where $f_{Q}(n)$ is defined to be  $f(n)$ if $P^-(n)>Q$ and 0 otherwise\footnote{Note that $S_j(e^u;f_{Q})=0$ for all $u\le\log Q$, since $j\ge1$.}. By relation\ \eqref{small-Q} and Lemma \ref{Lbound-1}(a), we find that
\[
S_0(e^u;f_{Q})  
	\ll_A e^u\frac{(\log Q)^{A-1}}{u^A}  \quad(u\ge4\log Q) .
\] 
Consequently, 
\[
S_j(e^u;f_{Q}) 
	=  O(e^{u/2}u^j) +  \int_{u/2}^u w^j d S_0(e^w;f_{Q})
	\ll_A \frac{e^u(\log Q)^{A-1}}{u^{A-j}}	\quad(u\ge8\log Q),
\]
by integration by parts. The above relation also holds trivially when $u\in[\log Q,8\log Q]$. Together with \eqref{complex-3-e6}, this implies that
\als{
\int_{\Re(s)=\sigma} |L_{Q}^{(j)}(\sigma+it,f)|^2  \frac{dt}{\sigma^2+t^2}
	&\ll_A (\log Q)^{2(A-1)} \int_{\log Q}^\infty \frac{du}{e^{2(\sigma-1)u} u^{2(A-j)} } \\
	&\le (\log Q)^{2(A-1)} \int_{\log Q}^\infty \frac{du}{u^{2(A-j)}}
		\ll (\log Q)^{2j-1} ,
}
since $2(A-j)\ge 2(A-k-1)\ge2$. Combining the above inequality with \eqref{complex-3-e5}, we obtain the estimate
\[
\sum_{ 2^m\le T}  \frac{I_{j,2^m}}{4^{m\left(1-\frac{k-j+1}{A-2}\right)}}
	\ll_A \frac{\log(2T)}{L_j^{1 - \frac{k-j+1}{A-2}} }\left(\frac{\log x}{\log Q}\right)^{2j-1}
		+  L_j^{\frac{k-j+2}{A-2}} .
\]
We choose $L_j=(\log x/\log Q)^{\frac{(A-2)(2j-1)}{A-1}}$ and note that $\log T\le \log T_2\ll_A\log(\log x/\log Q)$. So
\eq{complex-3-e4b}{
\sum_{ 2^m\le T  } \frac{I_{j,2^m}}{4^{m\left(1-\frac{k-j+1}{A-2}\right)}}
	\ll_A \left(\frac{\log x}{\log Q}\right)^{\frac{(2j-1)(k-j+2)}{A-1} }\left( \log \frac{\log x}{\log Q} \right) .
}
Finally, note that
\[
(k-j+2)(2j-1)  \le   \left(  \fl{ \frac {k}2 }   +1  \right)  \left( 2k - 2  \fl{ \frac{k}{2} } +  1   \right)
			=   \frac{(k+1)(k+2)}2 ,
\]
since $k-\lfloor k/2\rfloor+1$ is the nearest integer to $k/2+5/4$ (i.e. the point where the quadratic polynomial $t\to(k-t+2)(2t-1)$ is maximized). Consequently, \eqref{complex-3-e4b} becomes
\[
\sum_{ 2^m\le T }  \frac{I_{j,2^m}}{4^{m\left(1-\frac{k-j+1}{A-2}\right)}}
	\ll_A \left(\frac{\log x}{\log Q}\right)^{\frac{(k+1)(k+2)}{2(A-1)}} \left( \log \frac{\log x}{\log Q} \right).
\]
Inserting the above relation into \eqref{complex-3-e4} completes the proof of Proposition \ref{complex-prop-3}.

\medskip


It remains to prove relation \eqref{complex-3-e2}. First, we prove a pointwise bound for $(L'/L)^{(k)}(s,f)$. Fix $\tau\in[2,T]$ and let $t\in J_\tau\subset [-T,T]$, so that $Q_t\le x$. Note that relation \eqref{small-Qt} and Theorem \ref{Lbound-2}(a) imply that
\al{
| L_z^{(j)}(\sigma+it,f)|
	& \ll_A (\log z)^j \min \left\{ \frac{A-j}{A-j-1}, \log \frac{\log x}{\log z} \right\} \label{complex-der-0} \\
	&\ll_A \left\{ (\log z)  \left(  \log  \frac{\log x}{\log z}  \right)^{\frac{1}{k+1}} \right\}^j \nonumber ,
}
uniformly in $\sqrt{Q_t}\le z\le \sqrt{x}$ and $j\in \{1,\dots,k+1\}$. Thus, applying Lemma \ref{der-bound} with $F(s)=L(s,f)$, we deduce that
\eq{complex-der-1}{
\left| \left(\frac{L'}{L}\right)^{(k)} ( \sigma+it, f ) \right| 
	\ll_A \left( \frac{\log z}{\min\{1,|L_z(\sigma+it,f)|\}} \right)^{k+1}
		 \left( \log  \frac{\log x}{\log z} \right) .
}
For each $t\in J_\tau$ we choose $z=y_t$, where
\eq{yt}{
\log y_t = \frac{1}{2} \cdot \min  \left\{ \frac{\log Q_t}{ \min\{ |L_{Q_t}(\sigma+it,f)|,1\} }  , \log x \right\}.
}
Note that Lemma \ref{dist-l1} and Theorem \ref{complex-dist} imply that
\[
\log|L_{y_t} (\sigma+it,f)| = O(1) + \sum_{y_t<p\le x} \frac{\Re(f(p)p^{-it})}{p} \ll_\epsilon 1,
\]
that is to say, $|L_{y_t} (\sigma+it,f)|\asymp_\epsilon 1$. So we find that
\eq{complex-der-2}{
\left| \left(\frac{L'}{L}\right)^{(k)} ( \sigma+it, f ) \right| 
	\ll_{A,\epsilon} (\log y_t)^{k+1} \left( \log  \frac{\log x}{\log y_t} \right) .
}
Before we proceed further, we need to bound $y_t$ in terms of $N(x;T)$. Firstly, we claim that $\log y_t\asymp \log Q_t / |L_{Q_t}(\sigma+it,f)|$. Indeed, we have that $L_{Q_t}(\sigma+it,f)\ll 1$, by Lemma \ref{Lbound-2}(a) and relation \ref{small-Qt}, which implies that
\[
\log y_t \asymp \min  \left\{ \frac{\log Q_t}{|L_{Q_t}(\sigma+it,f)|}  , \log x \right\} .
\]
Moreover, since $Q_t\le x$ for $|t|\le T\le T_2$, we also have that
\[
\log|L_{Q_t} (\sigma+it,f)| = O(1) + \sum_{Q_t<p\le x} \frac{\Re(f(p)p^{-it})}{p} 
	\ge O(1) - \log\left(\frac{\log x}{\log Q_t}\right) ,
\]
where we used Lemma \ref{dist-l1} and Mertens' estimate on $\sum_{p\le t}1/p$. This proves our claim that $\log y_t\asymp \log Q_t / |L_{Q_t}(\sigma+it,f)|$. As a result, Lemma \ref{dist-l1} and the definition of $N(x;T)$ imply that
\eq{N(x;T)-alt}{
\frac{\log x}{\log y_t} 
	\asymp \frac{\log x}{\log Q_t} \cdot |L_{Q_t}(\sigma+it,f)| 
	&\asymp \exp\left\{\sum_{Q_t<p\le x} \frac{1+\Re(f(p)p^{-it}) }{p} \right\} \\
	&\ge \max\left\{ \frac{e^{N(x;T)}}{\log Q_t}, 1 \right\}    \quad(-T\le t\le T),
}
which provides the required relation between $y_t$ and $N(x;T)$.


Now, let $t_0\in J_\tau$ be such that $\log y_{t_0} = \max_{t\in J_\tau} \log y_t$, and set 
\[
\mathscr{A}_\tau=J_\tau \cap[t_0-1/\log Q_\tau,t_0+1/\log Q_\tau]
	\quad\text{and}\quad 
\mathscr{B}_\tau=J_\tau \setminus\mathscr{A}_\tau .
\] 
Theorem \ref{min-dist} and relation \eqref{small-Qt} imply that 
\eq{complex-max}{
\log y_t \asymp
	\begin{cases}
		\log y_{t_0} 	&\text{if}\ |t-t_0| \le 1 / \log y_{t_0}  ,\cr 
		|t-t_0|^{-1}		&\text{if}\ 1/\log y_{t_0} \le |t-t_0| \le 1 / \log Q_\tau  ,\cr 
		\log Q_\tau 	&\text{if}\ |t-t_0|\ge 1/\log Q_\tau, 
	\end{cases}
}
for every $t\in J_\tau$. The above estimate and relation \eqref{complex-der-2} imply that
\[
\int_{\mathscr{A}_\tau} \left| \left( \frac{L'}{L}\right)^{(k)}(\sigma+it,f)\right|^2  dt
	\ll_{A,\epsilon} ( \log y_{t_0} )^{2k+1} \left( \log \frac{\log x}{\log y_{t_0} } \right)  .
\]
Together with relation \eqref{N(x;T)-alt}, this implies that
\[
\int_{\mathscr{A}_\tau} \left| \left( \frac{L'}{L}\right)^{(k)}(\sigma+it,f)\right|^2  dt
	\ll_{A,\epsilon} \begin{cases} 
		\left( \frac{(\log x)(\log Q_{\tau})}{ e^{N(x;T)} }\right)^{2k+1} 
			\left( \log\frac{e^{N(x;T)}}{\log Q_\tau } \right)
		&\text{if}\ e^{N(x;T)}\ge 2\log Q_\tau,\cr
		(\log x)^{2k+1} &\text{if}\ e^{N(x;T)}< 2\log Q_\tau.
	\end{cases}
\]
In any case, since $Q_\tau\ge Q$ and $e^{N(x;T)}\ge 2\log Q$, we find that
\eq{complex-e2-A}{
\int_{\mathscr{A}_\tau} \left| \left( \frac{L'}{L}\right)^{(k)}(\sigma+it,f)\right|^2  dt
	\ll_{A,\epsilon} (N(x;T) - \log\log Q) \left( \frac{(\log x)(\log Q_\tau )}{e^{N(x;T)}} \right)^{2k+1}  ,
}
which is admissible. Finally, we bound the contributions to the integral on the left hand side of \eqref{complex-3-e2} from $t\in \mathscr{B}_\tau $. Note that
\eq{P- P+}{
\left( \frac{L'}{L} \right)^{(k)}(\sigma+it,f)
		-  \left( \frac{L_{Q}'}{L_{Q}} \right)^{(k)}(s,f)   
		&= \sum_{P^-(n)\le Q}\frac{f(n)\Lambda(n)(\log n)^k}{n^s} \\
		&= \sum_{P^+(n)\le Q}\frac{f(n)\Lambda(n)(\log n)^k}{n^s} .
}
for any $s\in\SC$ with $\Re(s)\ge1$. Moreover, 
\als{
\int_{J_\tau}\left|\sum_{P^+(n)\le Q}
		 \frac{ f(n) \Lambda(n)(\log n)^{k} }{ n^{\sigma+it} }  \right|^2  dt 
		 &\le 3 \int_{-\tau}^{\tau}\left|\sum_{P^+(n)\le Q}
		 \frac{ \Lambda(n)(\log n)^{k} }{ n^{\sigma+it} }  \right|^2  dt \\
		 &\ll_A (\log Q )^{2k+1} + \tau (\log V_\tau)^{2k+2},
}
by Lemmas \ref{mont} and \ref{lambda}. So we deduce that
\eq{complex-3-e7}{
\int_{\mathscr{B}_\tau } \left| \left( \frac{L'}{L} \right)^{(k)}(\sigma+it,f)
		-  \left( \frac{L_{Q}'}{L_{Q}} \right)^{(k)}(\sigma+it,f)   \right|^2  dt
	\ll_A (\log Q)^{2k+1} + \tau (\log V_\tau)^{2k+2}.
}
Next, Lemma \ref{der-lemma} implies that
\eq{complex-3-e8}{
\int_{\mathscr{B}_\tau } \left| \left(  \frac{L_{Q}'}{L_{Q}} \right)^{(k)} ( \sigma+it,f )  \right|^2  dt
	& \ll_A \sum_{j=1}^{k+1}\int_{\mathscr{B}_\tau} 
		\left|\frac{ L_{Q}^{ (j) } }{ L_{Q} }(\sigma+it,f) \right|^{\frac{2(k+1)}j}  dt.
}
Relation \eqref{complex-max} gives us that
\eq{complex-der-3}{
|L_{Q_t}(\sigma+it,f)| \asymp_\epsilon1 \quad( t\in\mathscr{B}_\tau).
}
So, when $j\in\{1,\dots,k\}\subset[1,A-2]$, then applying Lemma \ref{der-bound} with $z=Q$, $f(n)n^{-it}$ in place of $f$, $j$ in place of $k$, $Q_t$ in place of $Q$ and $M\asymp \log Q_t$, which is possible by relation \eqref{complex-der-0}, we deduce that
\eq{complex-3-e9a}{
\frac{L_Q^{ (j) } } { L_Q } (\sigma+it,f)	
	\ll_{A,\epsilon}  \left( \frac{\log Q_t}{\min\{1,|L_{Q_t}(\sigma+it,f)|\} }\right)^j
	\ll_{A,\epsilon} (\log Q_\tau)^j	\quad( t\in \mathscr{B}_\tau )  .
}
Consequently,
\eq{complex-3-e9}{
\int_{\mathscr{B}_\tau } \left| \frac{L_{Q}^{ (j) } } { L_{Q} } (\sigma+it,f) \right|^{ \frac{ 2(k+1) }{j} }dt  
	\ll_{A,\epsilon}  (\log Q_\tau)^{2(k-j+1)}   
		\int_{\mathscr{B_\tau}}  \left|  \frac{L_Q^{ (j) } } { L_Q} (\sigma+it,f)  \right|^2  dt,
}
by \eqref{complex-3-e9a} if $j\in\{1,\dots,k\}$ and trivially if $j=k+1$. It remains to bound the integral on the right hand side of \eqref{complex-3-e9}, which we perform in two different ways. Firstly, Lemmas \ref{mont} and \ref{lambda} imply the `trivial' bound 
\eq{complex-3-e10}{
\int_{\mathscr{B}_\tau }  \left|  \frac{L_Q^{(j)}}{L_Q}(\sigma+it,f) \right|^2 dt 
	\le 3  \int_{-\tau}^\tau  \left|  \frac{\zeta^{ (j) } }{\zeta}(\sigma+it)  \right|^2  dt 
	 \ll_A \tau (\log V_\tau)^{2j}  +    (\log x)^{2j-1}.
}
Finally, observe that 
\[
|L_Q(\sigma+it,f)| 
	\gg \frac{\log Q}{\log Q_t} |L_{Q_t}(\sigma+it,f)|
	\asymp_\epsilon \frac{\log Q}{\log Q_t} \asymp \frac{1}{\tau^{\frac{1}{A-2}} } \quad(t\in\mathscr{B}_\tau),
\] 
by \eqref{complex-der-3}, and thus
\eq{complex-3-e11}{
\int_{\mathscr{B}_\tau }  \left|  \frac{L_Q^{ (j)} }{ L_Q } (\sigma+it,f) \right|^2 dt
	\ll_\epsilon  \tau^{\frac{2}{A-2}}   \int_{J_\tau}  \left|  L_{Q}^{(j)}(\sigma+it,f)   \right|^2dt.
}
Combining relations \eqref{complex-3-e10} and \eqref{complex-3-e11}, we conclude that
\[
\int_{\mathscr{B}_\tau}  \left|  \frac{L_Q^{ (j)} }{L_Q}(\sigma+it,f) \right|^2  dt  
	\ll_{A,\epsilon}  (\log Q)^{2j-1} I_{j,\tau}  + \tau  (\log V_\tau)^{ 2j}    .
\]
The above estimate and relation \eqref{complex-3-e9} imply that
\als{
\int_{\mathscr{B}_\tau } \left| \frac{L_{Q}^{ (j) } } { L_{Q} } (\sigma+it,f) \right|^{ \frac{ 2(k+1) }{j} }dt  
	& \ll_{A,\epsilon} (\log Q_\tau)^{2(k-j+1)}  
		 \left\{ (\log Q)^{2j-1} I_{j,\tau}  + \tau  (\log V_\tau)^{ 2j}  \right\} \\
	&\ll_{A}  (\log Q)^{2k+1} \tau^{\frac{2(k-j+1)}{A-2}} I_{j,\tau}  
 		+ \tau (\log Q_\tau)^{2(k-j+1)} (\log V_\tau)^{ 2j}  ,
}
since $\log Q_\tau \asymp \tau^{\frac{1}{A-2}} \log Q$. Furthermore, note that
\[
(\log Q_\tau)^{2(k-j+1)} (\log V_\tau)^{ 2j}
	\ll (\log V_\tau)^{2k+2}+ (\log Q_\tau)^{2k}(\log V_\tau)^2
	\ll_A (\log V_\tau)^{2k+2}+  (\log Q_\tau)^{2k+1}.
\]
The above inequalities, together with \eqref{complex-3-e8}, yield the estimate
\als{
\int_{\mathscr{B}_\tau } \left| \left(  \frac{L_{Q}'}{L_{Q}} \right)^{(k)} ( \sigma+it,f )  \right|^2  dt
	& \ll_{A,\epsilon}  \tau \left\{ (\log V_\tau)^{2k+2}+  (\log Q_\tau)^{2k+1} \right\} \\
	&\quad  + ( \log Q)^{2k+1} \sum_{j=1}^{k+1} \tau^{\frac{2(k-j+1)}{A-2} }  I_{j,\tau}   .
}
Together with \eqref{complex-e2-A}, the above relation completes the proof of \eqref{complex-3-e2} and hence of Proposition \ref{complex-prop-3}.
\end{proof}


\begin{prop}\label{power-prop} Let $f:\SN\to[-1,1]$ be a completely multiplicative function which satisfies \eqref{power} for some $\delta\in(0,1/3)$ and some $Q\ge e^{1/\delta}$. For $x\ge 2$, $\sigma>1$ and $k\in\SN\cup\{0\}$, we have that
\als{
I_k(\sigma;\Lambda f)
	&\le k! \left( \frac{c \log Q}{L_Q(1,f) }  \right)^{  k + \frac{1}{2} }  + c^k k!^2 ,
}
where $c=c(\delta)$ is some positive constant.
\end{prop}


\begin{proof} All constants $c_1,c_2,\dots$ that appear below might depend on $\delta$, but no other parameters. Also, without loss of generality, we may assume that $\delta$ is small enough, so that $Q\ge 100$.

We follow an argument that is similar with the one leading to Proposition \ref{complex-prop-3}. As in that proof, our starting point is relation \eqref{plancherel}. We bound the right hand side of \eqref{plancherel} by breaking the range of integration into sets of the form $J_\tau=\{t\in\SR : \tau/2-1 \le |t|\le \tau-1\}$, for $\tau\in\{2^m: m\in\SN\}$. For every $\tau\ge2$, we claim that
\eq{power-prop-e1}{
\frac{1}{k!^2} \int_{J_\tau} \left|  \left(  \frac{L'}{L}  \right)^{(k)}( \sigma+it,f ) \right|^2 dt
	&\le  \tau  \left( \frac{c_1 \log Q}{L_Q(1,f)} \right)^{2k+1} +  \tau (c_1\log\tau )^{2k+2} \\
	&\quad+ (c_1\log Q)^{2k+1} \sum_{j=1}^{k+1} \int_{J_\tau}   \frac{|L_Q^{(j)} ( \sigma+it,f ) |^2}{j!^2 (\log Q)^{2j-1} } dt .
}
Before we prove this estimate, we show how to use it to deduce Proposition \ref{power-prop}. Clearly, if this relation holds, then combining it with relations \eqref{plancherel} and \eqref{complex-3-e6} we deduce that
\al{
\frac{I_k(\sigma;\Lambda f)^2}{k!^2} 
	&\ll \sum_{m=1}^\infty \frac{1}{k!^2  4^m} 
	 	\int_{J_{2^m}} \left|  \left(  \frac{L'}{L}  \right)^{(k)}( \sigma+it,f ) \right|^2 dt  \nn
	&\ll \sum_{m=1}^\infty \frac{1}{2^m} \left\{ \left( \frac{c_1 \log Q}{L_Q(1,f)} \right)^{2k+1} 
		+  (c_2m )^{2k+2} \right\}  \nn
	&\quad + (c_1\log Q)^{2k+1} \sum_{j=1}^{k+1} \int_{\SR}
			\frac{|L_Q^{(j)} ( \sigma+it,f ) |^2}{j!^2 (\log Q)^{2j-1} } \frac{dt}{\sigma^2+t^2} \nn
	&\ll \left( \frac{c_1 \log Q}{L_Q(1,f)} \right)^{2k+1} + c_3^{k+1} k!^2 
		 +  \sum_{j=1}^{k+1}  \frac{ c_1^{2k+1}(\log Q)^{2k-2j+2} }{j!^2} 
		 	\int_{\log Q}^\infty  \left|\frac{S_j(e^u;f_{Q})}{e^{\sigma u}} \right|^2  du ,
			 \label{power-prop-e2}
}
where $f_Q(n)= f(n)$ if $P^-(n)>Q$ and $f(n)=0$ otherwise. Fix $j\in\{1,\dots,k+1\}$. Relation\ \eqref{small-power} and our assumption that $Q\ge e^{1/\delta}$ imply that
\[
|S_0(e^u; f)| \le \frac{e^{u(1-1/\log Q ) }}{u^2}  \quad(u\ge \log Q).
\]
So Lemma \ref{Lbound-1}(a) yields that
\[
S_0(e^u;f_{Q})  
	\ll \frac{(\log Q)e^{u(1- 1/(2\log Q)  )} } {u^2} + \frac{e^{u(1-1/(2\log Q)) } }{\log Q}
	\ll \frac{e^{u(1-1/(2\log Q)) } }{\log Q}   \quad(u \ge 4\log Q) .
\]
This relation also holds trivially if $u\in[(\log Q)/2,4\log Q]$. So partial summation implies that
\als{
S_j(e^u;f_{Q}) 
	&=  O\left( \frac{e^{u/2}u^j}{\log Q} \right) +  \int_{u/2}^u w^j d S_0(e^w;f_{Q}) \\
	&=  O\left( \frac{e^{u/2}u^j}{\log Q} \right) + u^j S_0(e^u; f_Q) -  j \int_{u/2}^u w^{j-1} S_0(e^w;f_{Q}) dw \\
	&\ll \frac{e^{u(1-1/(2\log Q)) } u^j }{\log Q} 
		+ j u^{j-1} \int_{u/2}^u  \frac{e^{w(1-1/(2\log Q))} }{ \log Q} dw 
	\ll \frac{e^{u(1-1/(2\log Q)) } j u^j }{\log Q} ,
}
for all $u\ge \log Q$. Consequently,
\als{
\int_{\log Q}^\infty  \left|\frac{S_j(e^u;f_{Q})}{e^{\sigma u}} \right|^2  du
	\ll \frac{j^2}{(\log Q)^2} \int_{\log Q}^\infty \frac{u^{2j}}{e^{u/\log Q}} du
	\le j^2 (2j)! (\log Q)^{2j-1}  .
}
Inserting the above estimate into \eqref{power-prop-e2} completes the proof of Proposition \ref{power-prop}, since $L_Q(1,f)\ll 1$ by Lemma 4.4(b).

\medskip


It remains to prove relation \eqref{power-prop-e1}. First, note that, for any $x \ge Q^2$, our assumption that $f$ satisfies \eqref{small-power} and the argument leading to \eqref{Lbound-1-it} imply that
\als{
 \sum_{n\le x}f(n) n^{-it} 
 	&=O(\sqrt{x}) + \frac{1}{x^{it}} \sum_{n\le x}f(n) n^{-it}  
		+ \int_{\sqrt{x}}^x \frac{it}{u^{1+it}} \left( \sum_{n\le u}f(n) \right) du \nn
	&\ll \sqrt{x} +  \frac{ x^{1- \delta}}{ (\log x)^2} 
		+ |t| \int_{\sqrt{x}}^x \frac{du}{u^\delta} \nn
	&\le \sqrt{x} +  \frac{ x^{1- \delta}}{ (\log x)^2} 
		+ \frac{|t|}{x^{\delta/2}} \int_{\sqrt{x}}^x du
	\ll \frac{ (1+|t|) x^{1- \delta/2}}{(\log x)^2}  .
}
So if $x\ge q_t = \max\{Q^4,(|t|+3)^{4/\delta}\}$, then
\eq{small-qt}{
\sum_{n\le x}f(n) n^{-it}
	&\ll \frac{x^{1- \delta / 4 } }{(\log x)^2} 
	  \le  \frac{ x^{1- 1/ \log q_t  } }{(\log x)^2} .
}
Consequently, Lemma \ref{Lbound-2}(b) yields that
\eq{power-der-0}{
|L_{q_t}^{(j)}(\sigma+it, f) | \ll j! (c_1 \log q_t )^j \quad(j\in\SN\cup\{0\},\ t\in\SR).
}
Together with Lemma \ref{der-bound}, this implies that
\eq{power-der-1}{
\left| \left(\frac{L'}{L} \right)^{(k)}(\sigma+it, f) \right| 
	\le k!\left( \frac{ c_2 \log q_t }{ |L_{q_t}(\sigma+it,f) | } \right)^{k+1}
		 \quad(t\in\SR).
}
Moreover, combining Theorems \ref{min-dist} and \ref{real-complex} with our assumption that $f$ is real-valued and satisfies \eqref{small-power}, we find that
\eq{power-yt}{
\frac{\log q_t}{|L_{q_t}(\sigma+it,f)|}
	\asymp_\delta \begin{cases}
	(\log Q)/L_Q(1,f)  &\text{if}\ |t|\le L_Q(1,f)/\log Q, \cr
	1/|t| &\text{if}\ L_Q(1,f)/\log Q\le |t|\le 1/\log Q,\cr
	\log(Q+|t|) &\text{if}\ |t|\ge 1/\log Q.
	\end{cases}
}
So 
\eq{power-prop-e3}{
\int_{-1/\log Q}^{1/\log Q} \left| \left( \frac{L'}{L}\right)^{(k)}(\sigma+it,f)\right|^2  dt
	\ll k!^2 \left( \frac{c_4\log Q}{L_Q(1,f)} \right)^{2k+1} ,
}
which is admissible. Therefore it remains to bound the contributions to the integral on the left hand side of \eqref{power-prop-e1} from $t\in \mathscr{B}_\tau: = J_\tau\setminus[-1/\log Q,1/\log Q] $. First, note that if $\tau>Q$, then relations \eqref{power-der-1} and \eqref{power-yt} imply that
\eq{power-prop-e4}{
\int_{ \mathscr{B}_\tau} \left| \left( \frac{L'}{L}\right)^{(k)}(\sigma+it,f)\right|^2  dt
\le k!^2 \tau (c_5 \log \tau  )^{2k+2},
}
which is admissible. Finally, assume that $\tau\le Q$. As in \eqref{P- P+}, we have that
\als{
\left( \frac{L'}{L} \right)^{(k)}(\sigma+it,f)
		-  \left( \frac{L_{Q}'}{L_{Q}} \right)^{(k)}(s,f)   
		= \sum_{P^+(n)\le Q}\frac{f(n) \Lambda(n)(\log n)^k}{n^s} 
}
for any $s\in\SC$ with $\Re(s)\ge1$. Moreover, 
\als{
\frac{1}{k!^2}\int_{J_\tau}\left|\sum_{P^+(n)\le Q}
		 \frac{ f(n) \Lambda(n)(\log n)^{k} }{ n^{\sigma+it} }  \right|^2  dt 
		 &\le \frac{3}{k!^2} \int_{-2\tau}^{2\tau}\left|\sum_{P^+(n)\le Q}
		 \frac{ \Lambda(n)(\log n)^{k} }{ n^{\sigma+it} }  \right|^2  dt \\
		 &\le (c_6\log Q)^{2k+1} + \tau (c_6\log V_\tau)^{2k+2},
}
by Lemmas \ref{mont} and \ref{lambda}. So we deduce that
\eq{power-prop-e5}{
\frac{1}{k!^2}&\int_{\mathscr{B}_\tau } \left| \left( \frac{L'}{L} \right)^{(k)}(\sigma+it,f)
		-  \left( \frac{L_{Q}'}{L_{Q}} \right)^{(k)}(\sigma+it,f)   \right|^2  dt \\
	&\ll  (c_7\log Q)^{2k+1} + \tau (c_7\log V_\tau )^{2k+2}.
}
Next, Lemma \ref{der-lemma} implies that
\eq{power-prop-e6}{
\frac{1}{(k+1)!^2} \int_{\mathscr{B}_\tau } \left| \left(  \frac{L_{Q}'}{L_{Q}} \right)^{(k)} ( \sigma+it,f )  \right|^2  dt
	& \le \sum_{j=1}^{k+1}  \frac{4^{k+1} }{ j!^{\frac{2(k+1)} j } } 
		\int_{\mathscr{B}_\tau} \left|\frac{ L_{Q}^{ (j) } }{ L_{Q} }(\sigma+it,f) \right|^{\frac{2(k+1)}j}  dt.
}
Using relation \eqref{power-yt} and since $\log q_t\asymp_\delta \log Q$ for $t\in J_\tau$, by our assumption that $\tau\le Q$, we obtain the estimate
\[
|L_{Q}(\sigma+it,f)| \asymp_\delta 1 \quad( t\in\mathscr{B}_\tau).
\]
Together with relation \eqref{power-der-0}, this implies that
\als{
\int_{\mathscr{B}_\tau } \left| \frac{L_{Q}^{ (j) } } { L_{Q} } (\sigma+it,f) \right|^{ \frac{ 2(k+1) }{j} }dt  
	&\le  c_8^{k+1} (j!(\log Q)^j)^{\frac{2(k-j+1)}{j}}   
		\int_{\mathscr{B_\tau}}  \left|\frac{ L_{Q}^{ (j) } }{ L_{Q} }(\sigma+it,f) \right|^2 dt \\
	&\ll_\delta  c_8^{k+1} (j!(\log Q)^j)^{\frac{2(k-j+1)}{j}}   \int_{\mathscr{B_\tau}}  |  L_Q^{ (j) } (\sigma+it,f)  |^2  dt .
}
Inserting the above bound into \eqref{power-prop-e6}, we find that
\eq{power-prop-e7}{
\frac{1}{k!^2}  \int_{\mathscr{B}_\tau } \left| \left(  \frac{L_{Q}'}{L_{Q}} \right)^{(k)} ( \sigma+it,f )  \right|^2  dt
	& \le c_9^{k+1} \sum_{j=1}^{k+1} \frac{(\log Q)^{2(k-j+1)} }{j!^2}
		 \int_{J_\tau}  |  L_Q^{ (j) } (\sigma+it,f)  |^2  dt ,
}
which is admissible. Combining relations \eqref{power-prop-e3}, \eqref{power-prop-e4} and \eqref{power-prop-e7}, we obtain relation \eqref{power-prop-e1}, thus completing the proof of the proposition.
\end{proof}


\subsection{Completion of the proofs}\label{complex-completion}

We conclude this section with the proof of part (a) of Theorems \ref{complex} and \ref{power}.

\begin{proof}[Proof of Theorem \ref{complex}(a)] Recall the definition of $B$ from \eqref{B} and set $k=\lfloor A-2 \rfloor$ and
\begin{equation}\label{B'}
B' = \min\left\{k  + \frac{1}{2} - \frac{(k+1)(k+2)}{4(A-1)}, \frac{(A-2)(2k+1)}{2k+A-1} \right\} . 
\end{equation}
We claim that $B'\ge B$. Indeed, if $2<A<3$ or if $3\le A<4$, it is straightforward to check that $B'=(A-2)/(2A-2)=B$ or $B'=3(A-2)/(2A-2)=B$, respectively. Assume now that $A\ge4$. Since $A-2\ge k>A-3$, we have that
\[
k  + \frac{1}{2} - \frac{(k+1)(k+2)}{4(A-1)} 
	\ge k  + \frac{1}{2} - \frac{(k+1)(k+2)}{4(k+1)} = \frac{3k}{4} \ge \frac{3(A-3)}{4} \ge \frac{2A}{3} - 2,
\]
and
\[
\frac{(A-2)(2k+1)}{2k+A-1} \ge \frac{(A-2)(2A-5)}{3A-7} \ge \frac{2A}{3} - 2,
\]
which together imply that $B'\ge B$ in this last case too. Thus it suffices to prove Theorem \ref{complex} with $B'$ in place of $B$. 

Without loss of generality, we may assume that $Q$ is large enough in terms of $A$ and $\epsilon$. Let $x\ge Q$ and $T\ge1$, and set $L = e^{N(x;T)} / \log Q $ and $z= x^{1/(10B'\log L) } $. Note that $L\ll \log x/\log Q $, by Lemma \ref{dist-l2}, which is applicable by relation \eqref{small-Qt}. Moreover, we may assume that $L$ is large enough in terms of $\epsilon$, so that $Q^4\le z\le x^{1/8}$; otherwise, the desired result holds trivially. Set $f_z(n)=f(n)$ if $P^-(n)>z$ and $f_z(n)=0$ otherwise. Relation \eqref{small-Q}, Lemma \ref{Lbound-1}(a) and the choice of $z$ imply that, for all $w\in[x^{1/4},x]\subset[z^2,x]$, we have that
\eq{complex-pf-e5}{
S_0(w;f_z) 
	&\ll_A w\cdot \frac{(\log Q)^{A-2}\log z}{(\log w)^A} + \frac{w^{1-1/(2\log z)}}{\log z} \\
	&\ll_A w\cdot \frac{(\log Q)^{A-2}}{(\log x)^{A-1}} + \frac{w}{x^{1/(8\log z)} \log z} 
		\ll_A \frac{ w }{L^{B'} \log x} .
}
The above estimate and partial summation then give us that
\eq{complex-pf-e6}{
S_1(x; f_z) \ll_A \frac{x}{ L^{B'} } .
}
Moreover, since $f_z\log=f_z*f_z\Lambda$, by our assumption that $f$ is completely multiplicative, Dirichlet's hyperbola method yields
\al{
S_1(x; f_z)&= \sum_{dm\le t}\Lambda(d)f_z(d) f_z(m) \nn
	&= \sum_{d\le x^{3/4} } \Lambda(d) f_z(d) \left( S_0(x/d; f_z) - S_0(x^{1/4}; f_z) \right) 
	 + \sum_{m\le x^{1/4}}  f_z(m) S_0(x/m;\Lambda f_z)  . \label{complex-hyperbola}
}
The above formula, and relations \eqref{complex-pf-e5} and \eqref{complex-pf-e6}, yield that 
\[
 \sum_{m\le x^{1/4}}f_z(m)  S_0(x/m;\Lambda f_z) 
	\ll_A   \frac{ x }{L^{B'}}
		+ \sum_{ d\le x^{3/4} } \Lambda(d) \frac{x/d}{L^{B'}\log x}
\ll  \frac{x}{L^{B'}} .
\]
Since  $S_0(w;\Lambda f_z)=S_0(w;\Lambda f)+O(z\log w)$ 
for $w\ge1$, $f_z(m)=0$ for $m\in(1,z]$, and $z\le x^{1/8}$, we find that
\al{
S_0(x; \Lambda f) + \sum_{z<m\le x^{1/4}} f_z(m) S_0(x/m ; \Lambda f)
	&= \sum_{m \le x^{1/4}} f_z(m) S_0(x/m ; \Lambda f)  \nn
	&= \sum_{m \le x^{1/4}} f_z(m) S_0(x/m ; \Lambda f_z )  + O( x^{3/8} \log x) \nn
	&\ll_A \frac{x}{L^{B'} } .   \label{remove z}
}
Next, set $\Delta=x/L^{B'} \in [ x/(\log x)^{B'}, x/2] $ and note that 
\[
|S_0(x/m;\Lambda f)-S_0(t/m;\Lambda f)|\ll \frac{\Delta}{m}  \quad(x-\Delta\le t\le x,\ m\le x^{1/4}).
\] 
In addition, $\sum_{m\le x^{1/4}}|f_z(m)|/m\ll \log x/\log z\asymp \log L$. Thus
\eq{insert D}{
S_0(x;\Lambda f)
	&= - \sum_{z<m\le x^{1/4}} f_z(m)\cdot\frac{1}{\Delta}\int_{x-\Delta}^x S_0(t/m;\Lambda f)dt
		+  O_A\left(  \frac{x}{L^{B'}} + \Delta \log L \right) \\
	&= - \sum_{z<m\le x^{1/4}} f_z(m) m \cdot\frac{1}{\Delta}\int_{\frac{x-\Delta}{m}}^{\frac{x}{m}}
			S_0(t;\Lambda f)dt   +  O_A\left(  \frac{x\log L}{L^{B'}} \right) \\
	& =  \frac{-1}{\Delta}\int_{\frac{x-\Delta}{x^{1/4}}}^{\frac{x}{z}} S_0(t;\Lambda f)
		\left( \sum_{\frac{x-\Delta}t<m\le\frac xt} f_z(m) m  \right)  dt
	+   O_A\left(   \frac{x\log L}{L^{B'}} \right)  .
}
For every $t\in[\sqrt{x},x/z]\subset[(x-\Delta)/x^{1/4},x/z]$ we apply Lemma \ref{fund-lemma} with $D=(x/t)^{1/3}\ge z^{1/3}$ and $y=z^{1/9}$ to obtain the estimate
\eq{complex-sieve}{
\left|\sum_{\frac{x-\Delta}t<m\le\frac xt} f_z(m) m\right|
	&\le  \frac xt \sum_{\substack{\frac{x-\Delta}t<m\le\frac xt\\P^-(m)>z^{1/9}}}1
	\le  \frac xt \sum_{\frac{x-\Delta}t<m\le\frac xt}(\lambda^+*1)(m)   \\
	&=  \frac xt \sum_{d\le(x/t)^{1/3}}  \lambda^+(d)  \left(\frac{\Delta/t}d+O(1)\right) \\
	&\ll  \frac{\Delta x}{t^2\log z} + \left(\frac{x}{t} \right)^{4/3} 
		\asymp \frac{\Delta x}{t^2\log z}  \asymp \frac{\Delta x \log L}{t^2\log x}   ,
}
since $\Delta/t\ge\sqrt{x/t}$ for $t\le x/z$. Consequently,
\eq{complex-pf-e9}{
\frac{S_0(x;\Lambda f)}x
	\ll_A  \frac{\log L}{\log x}\int_{\sqrt{x}}^x\frac{|S_0(t;\Lambda f)|}{t^2}dt
		+   \frac{\log L}{L^{B'}}
}
We want to relate the above integral with an average involving $S_k(t;\Lambda f)$. By partial summation, we have that
\als{
S_0(t;\Lambda f) 
	&= O(\sqrt{t})+\int_{\sqrt{t}}^t\frac1{(\log u)^k}  dS_k(u;\Lambda f)
	=  O(\sqrt{t})+\frac{S_k(t;\Lambda f)}{(\log t)^k} 
		+  \int_{\sqrt{t}}^t\frac{k\,S_k(u;\Lambda f)}{u(\log u)^{k+1}}du \\	
	&\ll  \sqrt{t}+ \frac{2^k |S_k(t;\Lambda f)| }{(\log x)^k} 
		+ \frac{5^k}{(\log x)^{k+1}} \int_{\sqrt{t}}^t\frac{|S_k(u;\Lambda f)|}{u} du \\	
}
for all $t\in[\sqrt{x},x]$. So, if we set $\sigma=1+1/\log x$, then
\al{
\int_{ \sqrt{x} }^x \frac{|S_0(t;\Lambda f)|}{t^2} dt  
	&\ll \frac{1}{x^{1/4}} + \frac{2^k}{(\log x)^k} \int_{ \sqrt{x}}^x \frac{|S_k(t;\Lambda f)|}{t^2} dt
		+ \frac{5^k}{(\log x)^{k+1}} 
			\int_{ \sqrt{x}  }^x\int_{\sqrt{t}}^t\frac{ |S_k(u;\Lambda f)| }{ut^2}   du  dt \nn
	&\le \frac{1}{x^{1/4}} 
		+ \frac{2^k}{(\log x)^k}  \int_{  \sqrt{x} }^x \frac{ |S_k(t;\Lambda f)| }{t^2} dt
		+\frac{5^k}{(\log x)^{k+1}}  \int_{ \sqrt[4]{x} }^x \frac{ |S_k(u;\Lambda f)| }{u^2} du \nn
	&\ll \frac{1}{x^{1/4}} + \frac{5^k}{(\log x)^{k}} 	
		\int_{ \frac{\log x}{4}}^{\log x}   \frac{ |S_k(e^w;\Lambda f)| }{e^w} dw  \nn
	&\le \frac{1}{x^{1/4}} + \frac{5^kI_k(\sigma;\Lambda f)}{(\log x)^k} 
		 \left( \int_{ \frac{\log x}{4}}^{\log x}  e^{ \frac{2w}{\log x} } dw \right)^{\frac{1}{2}}
		\ll \frac{1}{x^{1/4}} + \frac{5^kI_k(\sigma;\Lambda f)}{(\log x)^{k-\frac{1}{2}}} ,
		\label{complex-CS}
}
by the Cauchy-Schwarz inequality. Inserting this estimate into \eqref{complex-pf-e9}, we deduce that
\[
\frac{S_0(x;\Lambda f)}x
	\ll_A  \frac{(\log L) I_k(\sigma; \Lambda f) }{(\log x)^{k+ \frac{1}{2}}}
		+   \frac{\log L}{L^{B'}} .
\]
So estimating $I_k(\sigma;\Lambda f)$ by Proposition \ref{complex-prop-3} implies that
\[
\frac{ S_0(x;\Lambda f) }x \ll_A \frac{(\log L)^2}{L^{B'}}  +  \frac{1}{T},
\]
and Theorem \ref{complex}(a) follows.
\end{proof}


\begin{proof}[Proof of Theorem \ref{power}(a)] We follow a similar argument with the one leading to Theorem \ref{complex}(a). Without loss of generality, we may assume that $Q$ is large enough. Let $x\ge Q$ and $k\in\SN$, and set
\[
L = \min\left\{ \sqrt{\log x} ,  \frac{L_Q(1,f) \log x }{\log Q} \right\}
\quad\text{and}\quad
z = \max\left\{ Q^4 , x^{1/ L } \right\}.
\]
We will show that 
\[
S_0(x; L) \ll x e^{-cL},
\]
for some constant $c=c(\delta)$. The theorem then follows, since $\eta = L_Q(1,f) \asymp (1-\beta)\log Q$, by Theorem \ref{siegel}.

Note that $L_Q(1,f)\ll1$, by \eqref{power-der-0}. So $L\ll \log x/\log Q$ and consequently  $\log z\asymp (\log x)/L$. Moreover, we may assume that $L$ is large enough, so that $z\le x^{1/8}$; otherwise, the desired result holds trivially. Set $f_z(n)=f(n)$ if $P^-(n)>z$ and $f_z(n)=0$ otherwise. Relation \eqref{small-power} and Lemma \ref{Lbound-1}(a) imply that 
\[
S_0(w; f_z)   
	\ll \frac{(\log z) w^{1- 1/(2\log z)  }}{(\log w)^2} +  \frac{w^{1- 1/(2\log z) } }{\log z}
	\ll  \frac{w^{1- 1/ (2\log z)  } } {\log z}   \ll \frac{w^{1- 1/(3\log z)} }{\log w}   \quad(w\ge z).
\]
Together with partial summation, this gives us that
\[
S_1(w; f_z)   
	\ll w^{1- 1/(3\log z) }   \quad(w\ge z).
\]
Inserting the above estimates into \eqref{complex-hyperbola}, we find that
\[
\sum_{m\le x^{1/4}} f_z(m)S_0(x/m;\Lambda f_z)
	\ll x^{1- 1/(3\log z) }
		+\frac{x^{1- 1/(3\log z)} } {\log x}   \sum_{ d\le x^{3/4} } \frac{ \Lambda(d) }{ d^{1- 1/(3\log z)} }  
	\ll x^{1- 1/(12\log z)}   ,
\]
by Chebyshev's estimate $\sum_{d\le t}\Lambda(d)\ll t$ and partial summation. So following the argument leading to \eqref{remove z}, we deduce that
\[
S_0(x;\Lambda f) = - \sum_{z<m\le x^{1/4}} f_z(m) S_0(x/m; \Lambda f) 
	+ O \left( x^{1- 1/(12\log z) }  	\right) .
\]
Set $\Delta=x/ e^{c_1 L} \in [ x/ e^{c_1\sqrt{\log x}} ,  x/2]$, for some constant $c_1>0$. Note that $\Delta\ge x/\sqrt{z}$, provided that $c_1$ is small enough, since $\log z \asymp (\log x)/L \ge L$. Thus arguing as in \eqref{insert D} and then as in \eqref{complex-sieve} implies that
\als{
S_0(x;\Lambda f)
	& =  \frac{-1}{\Delta}\int_{\frac{x-\Delta}{x^{1/4}}}^{\frac{x}{z}} S_0(t;\Lambda f)
		\left( \sum_{\frac{x-\Delta}t<m\le\frac xt} f_z(m) m \right) dt
	+   O \left( x^{1- 1/(12\log z) }  + L\Delta \right) \\
	& \ll  \frac{x}{\log z} \int_{\sqrt{x}}^x\frac{|S_0(t;\Lambda f)|}{t^2}dt
		+   \frac{x}{ e^{c_2L}}
		\asymp \frac{x L}{\log x} \int_{\sqrt{x}}^x\frac{|S_0(t;\Lambda f)|}{t^2}dt
		+   \frac{x}{ e^{c_2L}},
}
since $\Delta/t\ge\sqrt{x/t}$ for $t\le x/z$. Combining the above estimate with \eqref{complex-CS} and Proposition \ref{power-prop}, we deduce that
\als{
\frac{S_0(x;\Lambda f)}x
	&\ll  L \cdot \frac{ 5^k I_k(\sigma; \Lambda f) }{(\log x)^{k+ \frac{1}{2}}}   +    e^{ - c_2L} 
		\ll L \cdot \left\{ \left( \frac{c_3k\log Q}{L_Q(1,f)\log x }  \right)^{  k + \frac{1}{2} }  
		+ \left( \frac{c_3k^2}{\log x} \right)^{k+ \frac{1}{2}}  \right\} 
		 +  e^{- c_2L}   ,
}
for some constant $c_3=c_3(\delta)\ge1$, where $\sigma=1+1/\log x$. Choosing $k = \fl{L/(e c_3)}$ in the above inequality, so that
\[
\frac{c_3k\log Q}{L_Q(1,f)\log x } \le \frac{L\log Q}{e L_Q(1,f)\log x} \le \frac{1}{e}
	\quad\text{and}\quad
\frac{c_3 k^2}{\log x}\le \frac{L^2}{e^2c_3\log x} \le \frac{1}{e^2},
\]
we conclude that
\[
\frac{S_0(x; \Lambda f)}{x} \ll L\cdot e^{- L/(ec_3)  } + e^{-c_2 L},
\]
which proves Theorem \ref{power}(a).
\end{proof}


\section*{Acknowledgements}
I would like to thank Andrew Granville for many fruitful conversations on the subject of multiplicative functions, and for several comments and suggestions, particularly, for drawing my attention to the proof of the prime number theorem given in \cite[pp. 40-42]{IK}, as well as for discussions that led to an improvement of Theorems \ref{small}(b) and \ref{complex-dist}. Moreover, I am grateful to both an anonymous referee and Andrew Granville, who suggested the idea of using the convolution $1*f*1*\overline{f}$ in the proof of part (b) of Theorem \ref{small}. The previous proof, which was similar to the proof of part (a), was substantially more complicated and yielded weaker results. Finally, I would like to thank K. Soundararajan and G\'erald Tenenbaum for some helpful comments.

This paper was largely written while I was a postdoctoral fellow at the Centre de recherches math\'ematiques at Montr\'eal, which I would like to thank for the financial support.


\bibliographystyle{alpha}

\end{document}